\documentclass{amsart}

\usepackage{amsthm}
\usepackage{amssymb}
\usepackage{color}
\usepackage{enumerate}
\usepackage[all]{xy}
\usepackage{verbatim}

\usepackage{tikz-cd} 
\usepackage{hyperref}

\SelectTips{cm}{}
\calclayout 

\numberwithin{equation}{section}  

\theoremstyle{plain}

\newtheorem{lemma}{Lemma}[section]
\newtheorem{theorem}[lemma]{Theorem}
\newtheorem*{theorem*}{Theorem}
\newtheorem*{corollary*}{Corollary}
\newtheorem{proposition}[lemma]{Proposition}
\newtheorem{corollary}[lemma]{Corollary}

\theoremstyle{definition}
\newtheorem{definition}[lemma]{Definition}
\newtheorem{example}[lemma]{Example}
\newtheorem{notation}[lemma]{Notation}
\newtheorem*{ack}{Acknowledgements}

\theoremstyle{remark} 
\newtheorem{remark}[lemma]{Remark}

\newcommand{\coind}{\operatorname{\mathrm{coind}}}

\newcommand{\cosupp}{\operatorname{cosupp}}
\renewcommand{\dim}{\operatorname{dim}}
\newcommand{\End}{\operatorname{End}}
\newcommand{\Ext}{\operatorname{Ext}}

\newcommand{\height}{\operatorname{ht}}
\newcommand{\Hom}{\operatorname{Hom}}
\newcommand{\sHom}{\underline{\Hom}}
\newcommand{\id}{\operatorname{id}}

\newcommand{\ind}{\operatorname{\mathrm{ind}}}
\newcommand{\Inj}{\operatorname{\mathsf{Inj}}}
\newcommand{\KacInj}[1]{\mathsf K_{\mathrm{ac}}(\Inj #1)}
\newcommand{\Ker}{\operatorname{Ker}}
\newcommand{\KInj}[1]{\mathsf K(\Inj #1)}

\newcommand{\Loc}{\operatorname{\mathsf{Loc}}}
\renewcommand{\mod}{\operatorname{\mathsf{mod}}}
\newcommand{\Mod}{\operatorname{\mathsf{Mod}}}

\newcommand{\Proj}{\operatorname{Proj}}

\newcommand{\res}{\operatorname{res}}

\newcommand{\Spec}{\operatorname{Spec}}

\newcommand{\StMod}{\operatorname{\mathsf{StMod}}}
\newcommand{\stmod}{\operatorname{\mathsf{stmod}}}

\newcommand{\supp}{\operatorname{supp}}
\newcommand{\pisupp}{\pi\text{-}\supp}
\newcommand{\picosupp}{\pi\text{-}\cosupp}

\newcommand{\ges}{{\scriptscriptstyle\geqslant}}
 
\newcommand{\kos}[2]{{#1}/\!\!/{#2}} 
\newcommand{\les}{{\scriptscriptstyle\leqslant}}

\newcommand{\da}{{\downarrow}}
\newcommand{\lra}{\longrightarrow}

\newcommand{\xra}{\xrightarrow}

\newcommand{\bik}{Benson/Iyengar/Krause}

\newcommand{\mcB}{\mathcal{B}}

\newcommand{\mcE}{\mathcal{E}}
\newcommand{\mcG}{\mathcal{G}}

\newcommand{\mcV}{\mathcal{V}}
\newcommand{\mcW}{\mathcal{W}} 
 
\newcommand{\mcZ}{\mathcal{Z}}

\newcommand{\sfb}{\mathsf b}

\newcommand{\sfg}{\mathsf g}

\newcommand{\sfu}{\mathsf u}

\newcommand{\sfC}{\mathsf C}
\newcommand{\sfD}{\mathsf D}

\newcommand{\sfT}{\mathsf T} 
\newcommand{\sfU}{\mathsf U}

\newcommand{\bbG}{\mathbb G}
\newcommand{\bbP}{\mathbb P}
 
\newcommand{\bbZ}{\mathbb Z} 

\newcommand{\bsa}{\boldsymbol{a}} 
\newcommand{\bsb}{\boldsymbol{b}} 
 
\newcommand{\bsp}{\boldsymbol{p}} 
\newcommand{\bsq}{\boldsymbol{q}} 
 
\newcommand{\bst}{\boldsymbol{t}} 
\newcommand{\bsx}{\boldsymbol{x}} 
 
\newcommand{\bsz}{\boldsymbol{z}}

\newcommand{\fa}{\mathfrak{a}} 

\newcommand{\fm}{\mathfrak{m}} 
\newcommand{\fp}{\mathfrak{p}}
\newcommand{\fq}{\mathfrak{q}}

\newcommand{\one}{\mathbf 1}

\newcommand{\gam}{\varGamma}

\newcommand{\GL}{\operatorname{GL}\nolimits}
\newcommand{\SL}{\operatorname{SL}\nolimits}

\newcommand{\Gr}{\mcG_{(r)}}

\title[Stratification]{Stratification for module categories of \\ finite group schemes}

\author[Benson, Iyengar, Krause, and Pevtsova]{Dave Benson, Srikanth B. Iyengar, Henning Krause \\ and Julia Pevtsova}

\address{Dave Benson \\ 
Institute of Mathematics\\ 
University of Aberdeen\\ 
King's College\\ 
Aberdeen AB24 3UE\\ 
Scotland U.K.}

\address{Srikanth B. Iyengar\\ 
Department of Mathematics\\
University of Utah\\ 
Salt Lake City, UT 84112\\ 
U.S.A.}

\address{Henning Krause\\ 
Fakult\"at f\"ur Mathematik\\ 
Universit\"at Bielefeld\\ 
33501 Bielefeld\\ 
Germany.}

\address{Julia Pevtsova\\ 
Department of Mathematics\\ 
University of Washington\\ 
Seattle, WA 98195\\ 
U.S.A.}

\begin{document}

\begin{abstract} 
The tensor ideal localising subcategories of the stable module category of all,  
including infinite dimensional, representations of a finite group scheme over 
a field of positive characteristic are classified.  Various applications concerning 
the structure of the stable module category and the behavior of support and 
cosupport under restriction and induction are presented.
\end{abstract}

\thanks{SBI and JP were partly supported by NSF grants DMS 1503044  and DMS 0953011, respectively.}

\date{\today}

\maketitle

\setcounter{tocdepth}{1}
\tableofcontents

\section*{Introduction}
This paper is about the representation theory of finite group schemes over a field $k$ of positive characteristic. A finite group scheme $G$ is an affine group scheme whose coordinate algebra is a finite dimensional vector space over $k$. In that case, the linear dual of the coordinate algebra, called the \emph{group algebra} of $G$, is a finite dimensional cocommutative Hopf  algebra, whose representation theory is equivalent to that of $G$. This means that all our results can be restated for finite dimensional cocommutative Hopf 
$k$-algebras, but we adhere to the geometric language.

Examples of finite group schemes include finite groups, restricted enveloping algebras of finite dimensional $p$-Lie algebras, and Frobenius kernels of algebraic groups. The representation theory of finite group schemes over $k$ is often wild, even in such small cases as the finite group $\bbZ/3 \times \bbZ/3$ over a field of characteristic three, or the $3$-dimensional Heisenberg Lie algebra. In constructive terms, this means that it is not possible to classify the finite dimensional indecomposable modules. One  thus has to find better ways to organise our understanding of the structure of the module category of $G$. Developments in stable homotopy theory and algebraic geometry suggest a natural extension of the process of building modules up to direct sums. Namely, in addition to (possibly infinite) direct sums and  summands, one allows taking syzygies (both positive and negative), extensions,  and tensoring  (over $k$) with simple $G$-modules. We say $M$ is \emph{built out} of $N$ if $M$ can be constructed out of $N$ using these operations.  What follows is one of the main results of this work.

\begin{theorem*}[Corollary~\ref{co:modules}]
Let $M$ and $N$ be non-zero $G$-modules. One can build $M$ out of $N$ if (and only if) there is an inclusion $\pisupp_{G}(M)\subseteq \pisupp_{G}(N)$.
\end{theorem*}

Here $\pisupp_{G}(M)$ denotes the $\pi$-support of $M$ introduced by Friedlander and Pevtsova~\cite{Friedlander/Pevtsova:2007a}, and recalled in Section~\ref{se:pi-points}. It is a subset of the space of $\pi$-points of $G$ and the latter is homeomorphic to $\Proj H^*(G,k)$. Recall that $H^{*}(G,k)$, the cohomology algebra of $G$, is a finitely generated graded-commutative $k$-algebra, by a result of Friedlander and Suslin~\cite{Friedlander/Suslin:1997a}. Thus, $\pisupp_{G}(M)$ may be seen as an algebro-geometric portrait of $M$ and the gist of the  theorem  is that this is fine enough to capture at least homological aspects of the $G$-module $M$. 

The proper context for the result above is $\StMod G$, the stable module category of $G$, and $\KInj{G}$, the homotopy category of
complexes of injective $G$-modules. These are compactly generated triangulated categories that inherit the tensor product of $G$-modules. We deduce Corollary~\ref{co:modules} from an essentially equivalent statement, Theorem~\ref{th:stratification}, that gives a classification of the tensor ideal localising subcategories of $\StMod G$. See Corollary~\ref{co:stratification} for a version dealing with $\KInj G$.

There are many known consequences of such classification results. One is a proof of the ``Telescope Conjecture'' for $\StMod G$ in Theorem~\ref{th:smashing}. Another is a characterisation of when there is a nonzero map in $\StMod G$ between $G$-modules $M$ and $N$, in terms of the $\pi$-support of $M$ and the $\pi$-cosupport of $N$ described further below. This last result is a generalisation, to not necessarily finite dimensional modules, of the fact that when $G$ is unipotent, and $M$ and $N$ are finite dimensional $G$-modules, one has
\[
\Ext^{i}_{G}(M,N)=0 \quad \text{for $i\gg 0$} \qquad \iff \qquad  \Ext^{i}_{G}(N,M)=0 \quad\text{for $i\gg 0$}\,.
\]
See Theorem~\ref{th:last} and the remark following that result. Over a general ring $R$, even one that is self-injective as the group algebra of a finite group scheme is, there is no correlation between $\Ext^{*}_{R}(M,N)$ and $\Ext^{*}_{R}(N,M)$, so this symmetry in the vanishing of Ext is surprising. This phenomenon was first discovered by Avramov and Buchweitz~\cite{Avramov/Buchweitz:2000a} when $R$ is a (commutative) complete intersection ring, and is now understood to be related to the classification of localising subcategories~\cite{Benson/Iyengar/Krause:2011a}.

When $M$ and $N$ are finite dimensional, $M$ is built of out $N$ only if it is \emph{finitely} built out of $N$, meaning that one needs only finite direct sums in the building process.  Consequently, the results mentioned in the preceding paragraph yield a classification of the tensor ideal thick subcategories of $\stmod G$ and $\sfD^{\sfb}(\mod G)$, the stable module category and the bounded derived category, respectively, of finite dimensional modules. This is because these categories are equivalent to the subcategories of  compact objects of $\StMod G$ and $\KInj{G}$, respectively.

\subsection*{A brief history}
The genesis of such results is a classification theorem of Devinatz,
Hopkins, and Smith for the stable homotopy category of finite spectra
\cite{Devinatz/Hopkins/Smith:1988a}. Classification theorems for other
``small'' categories followed: see Hopkins~\cite{Hopkins:1987a} and
Neeman~\cite{Neeman:1992a} for perfect complexes over a commutative
noetherian ring; Thomason~\cite{Thomason:1997a} for perfect complexes
of sheaves over a quasi-compact, quasi-separated scheme; Benson,
Carlson, and Rickard~\cite{Benson/Carlson/Rickard:1997a} for finite
dimensional modules of a finite group, as well as many more recent
developments. Our results cover not only finite dimensional modules,
but also the ``big'' category of all $G$-modules, so the closest
precursor is Neeman's classification~\cite{Neeman:1992a} for all
complexes over a commutative noetherian ring. An analogous statement
for group schemes arising from finite groups is proved
in~\cite{Benson/Iyengar/Krause:2011b}. Theorem~\ref{th:stratification}
is new for all other classes of finite groups schemes.

\subsection*{Structure of the proof}
Arbitrary finite group schemes  lack many of the structural properties of finite groups, as we explain further below. Consequently the methods we use in this work are fundamentally different from the ones that lead to the successful resolution of the classification problem for finite groups in ~\cite{Benson/Iyengar/Krause:2011b}.  In fact, our proof of Theorem~\ref{th:stratification} provides another proof for the case of finite groups. The two new ideas developed and exploited in this work are that of $\pi$-cosupport of a $G$-module introduced in \cite{Benson/Iyengar/Krause/Pevtsova:2015a}, and a technique of reduction to closed points that enhances a local to global principle from \cite{Benson/Iyengar/Krause:2011a, Benson/Iyengar/Krause:2011b}.

For a finite group $G$, the proof of the classification theorem given in \cite{Benson/Iyengar/Krause:2011b} proceeds by a reduction to the case of elementary abelian groups. This hinges  on Chouinard's theorem that a $G$-module is projective if and only if its restriction 
to all elementary abelian subgroups of $G$ is projective.  Such a reduction is an essential step also in a second proof of the classification theorem for finite groups described in~\cite{Benson/Iyengar/Krause/Pevtsova:2015a}. For general finite group schemes there is no straightforward replacement for a detecting family of subgroups akin to elementary abelian subgroups of finite groups: for any such family one needs to allow scalar extensions. See Example~\ref{ex:sl2} and the discussion around Corollary~\ref{co:Chouinard}.  

The first crucial step in the proof of the classification theorem is to verify that $\pi$-support detects projectivity: 

\begin{theorem*}[Theorem~\ref{th:main}]
Any  $G$-module $M$ with $\pisupp_G(M)=\varnothing$ is projective.
\end{theorem*}

The essence of this detection theorem is that projectivity is detected locally on $\pi$-points of $G$.  Thus the geometric portrait of modules given by $\pi$-support is a faithful invariant on $\StMod G$.  This result is an ultimate generalisation, to arbitrary finite group schemes and to infinite dimensional modules, of Dade's Lemma~\cite{Dade:1978b} that the projectivity of a finite dimensional module over an elementary abelian group is detected on cyclic shifted subgroups.  The proof of the detection theorem builds on the work of many authors. For finite groups, Dade's Lemma was generalised to infinite dimensional modules by Benson, Carlson, and Rickard~\cite{Benson/Carlson/Rickard:1997a}.  For connected finite groups schemes the analogue of Dade's Lemma is that projectivity can be detected by restriction to  one-parameter subgroups, and was proved in a series of papers by Suslin, Friedlander, and Bendel~\cite{Bendel/Friedlander/Suslin:1997b},  Bendel~\cite{Bendel:2001a}, and Pevtsova~\cite{Pevtsova:2002a,Pevtsova:2004a}.

There is a flaw in the proof of the detection theorem, Theorem~\ref{th:main}, given in \cite[Theorem~5.3]{Friedlander/Pevtsova:2007a},  as we explain in Remark~\ref{re:fpnot}. For this reason, Part~\ref{part:detection} of this paper is devoted to a proof  of this result. Much of the argument is already available in the literature but is spread across various places. The new idea that allowed us to repair the proof appears in a ``subgroup reduction principle", Theorem~\ref{th:principle}, which also led to some simplifications of the known arguments.  

As a consequence of the $\pi$-support detection theorem we prove:

\begin{theorem*}[Theorem~\ref{th:pisupp=bik}]
For any $G$-module $M$ there is an equality
\[
\pisupp_{G}(M) = \supp_{G}(M)\,.
\]
\end{theorem*}
Here $\supp_{G}(M)$ is the support of $M$ defined in \cite{Benson/Iyengar/Krause:2008a} using the action of $H^{*}(G,k)$ on $\StMod G$,  recalled in Section~\ref{se:bik}. This allows us to apply the machinery developed in \cite{Benson/Iyengar/Krause:2008a, Benson/Iyengar/Krause:2011a, Benson/Iyengar/Krause:2011b}. The first advantage that we  reap from this is the following \emph{local to global principle}: for the desired classification it suffices to verify that for each point $\fp$ in $\Proj H^{*}(G,k)$ the subcategory of $\StMod G$ consisting  of modules with support in $\fp$ is \emph{minimal}, in that it has no proper tensor ideal localising subcategories. The latter statement is equivalent to the category $\StMod G$ being {\em stratified} by $H^{*}(G,k)$ as we explain towards the end of Section~\ref{se:bik}. This is tantamount to proving that when $M,N$  are $G$-modules whose support equals $\{\fp\}$, the $G$-module  of homomorphisms $\Hom_{k}(M,N)$ is not projective.

When $\fp$ is a closed point in $\Proj H^{*}(G,k)$, we verify this by using a new invariant of $G$-modules 
called $\pi$-cosupport introduced in \cite{Benson/Iyengar/Krause/Pevtsova:2015a}, and recalled in 
Section~\ref{se:pi-points}. Its relevance to the problem on hand stems from the equality below; see 
Theorem~\ref{th:tensor-and-hom-pi}.
\[
\picosupp_G(\Hom_k(M,N)) =\pisupp_G(M)\cap\picosupp_G(N)
\]

The minimality for a general point $\fp$ in $\Proj H^{*}(G,k)$ is established by a reduction to the case  of closed points. To this end, in Section~\ref{se:passage-to-closed-points} we develop a technique that  mimics the construction of generic points for irreducible algebraic varieties in the classical theory of Zariski and Weil. The results from commutative algebra that are required for this last step are established 
in Section~\ref{se:generic-points}.  The ideas in these sections are an elaboration of the local  to global principle alluded to above. 

\subsection*{Applications}
One of the many known consequences of Theorem~\ref{th:stratification}  is a classification  of the tensor-ideal thick subcategories of $\stmod G$, anticipated in \cite{Hovey/Palmieri:2001a}, and of $\sfD^{\sfb}(\mod G)$. This was mentioned earlier and is the content of Theorem~\ref{th:thick} and Corollary~\ref{co:stratification}. A few others are described in Section~\ref{se:stratification}. The results in this work also yield a precise criterion for deciding when there is a nonzero map between $G$-modules $M$ and $N$, at least when $G$ is unipotent; see Theorem~\ref{th:last}. Further applications specific to the context of finite group schemes are treated in Section~\ref{se:cosupport}. These include a proof that, akin to supports, the $\pi$-cosupport of a $G$-module coincides with its cosupport in the sense of \cite{Benson/Iyengar/Krause:2012b}: 

\begin{theorem*}[Theorem~\ref{th:picosupp=bik}]
For any $G$-module $M$ there is an equality
\[
\picosupp_{G}(M) = \cosupp_{G}(M)\,.
\]
\end{theorem*}
This in turn is used to track support and cosupport under restriction and induction for subgroup schemes;  see Proposition~\ref{pr:ind}.  That the result above is a consequence of the classification theorem also illustrates a key difference between the approach developed in this paper and the one in \cite{Benson/Iyengar/Krause/Pevtsova:2015a} where we give a new proof of the classification theorem for finite groups. There we prove that $\pi$-cosupport coincides with cosupport, using linear algebra methods and special properties of finite groups, and deduce the classification theorem from it. In this paper our route is the opposite: we have to develop a new method to prove classification and then deduce the equality of cosupports from it. See also Remark~\ref{re:cosupp}. 

The methods developed in this work have led to other new results concerning the structure of the stable module category of a finite
group scheme, including a classification of its Hom closed  \emph{colocalising} subcategories~\cite{Benson/Iyengar/Krause/Pevtsova:2015c}, and to a type of local Serre duality theorem for $\StMod G$; see~\cite{Benson/Iyengar/Krause/Pevtsova:2016a}.

\part{Recollections}
\label{part:recollections} There have been two, rather different,
approaches to studying representations of finite groups and finite
groups schemes using geometric methods: via the theory of $\pi$-points
and via the action of the cohomology ring on the stable category. Both
are critical for our work.  In this part we summarise basic
constructions and results in the two approaches.

\section{$\pi$-support and $\pi$-cosupport}
\label{se:pi-points}
In this section we recall the notion of $\pi$-points for finite group schemes. The primary references are the papers of Friedlander and
Pevtsova \cite{Friedlander/Pevtsova:2005a,Friedlander/Pevtsova:2007a}. We begin by summarising basic properties of modules over affine group schemes; for details we refer the reader to Jantzen \cite{Jantzen:2003a} and Waterhouse \cite{Waterhouse:1979a}.

\subsection*{Affine group schemes and their representations} Let $k$
be a field and $G$ a group scheme over $k$; this work concerns only
\emph{affine} group schemes. The coordinate ring of $G$ is denoted
$k[G]$; it is a commutative Hopf algebra over $k$. One says that $G$
is \emph{finite} if $k[G]$ is finite dimensional as a $k$-vector
space.  The $k$-linear dual of $k[G]$ is then a cocommutative Hopf
algebra, called the \emph{group algebra} of $G$, and denoted $kG$.  A
finite group scheme $G$ over a field $k$ is \emph{connected} (or
\emph{infinitesimal}) if its coordinate ring $k[G]$ is local; it is
\emph{unipotent} if its group algebra $kG$ is local.

\begin{example}[Finite groups]
  \label{ex:finite-groups} 
A finite group $G$ defines a finite group scheme over any field $k$: The group algebra $kG$ is a finite dimensional cocommutative Hopf algebra, hence its dual is a commutative Hopf algebra and so defines a group scheme over $k$; it is also denoted $G$.  A finite group $E$ is \emph{elementary abelian} if it is isomorphic to $(\bbZ/p)^{r}$, for some prime number $p$.  The integer $r$ is then the \emph{rank} of $E$. Over a field $k$ of characteristic $p$, there are isomorphisms of $k$-algebras
\[ 
k[E]\cong k^{\times {p^{r}}}\quad\text{and}\quad kE\cong k[z_1,\dots,z_{r}]/(z_1^p,\dots,z_{r}^p).
\]
\end{example}

\begin{example}[Frobenius kernels]
\label{ex:frobenius-kernels} 
Let $k$ be a field of positive characteristic $p$ and $f\colon k\to k$ its Frobenius endomorphism; thus $f(\lambda)=\lambda^p$.  The \emph{Frobenius twist} of a commutative $k$-algebra $A$ is the base change $A^{(1)} := k\otimes_f A$ over the Frobenius map. There is a $k$-linear algebra map $F_A\colon A^{(1)}\to A$ given by $F_A(\lambda \otimes a) = \lambda a^p$. 

If $G$ is a group scheme over $k$, then the Frobenius twist $k[G]^{(1)}$ is again a Hopf algebra over $k$ and therefore defines another group scheme $G^{(1)}$ called the \emph{Frobenius twist} of $G$. The algebra map $F_{k[G]} \colon k[G^{(1)}] = k[G]^{(1)}\to k[G]$ induces the Frobenius map of group schemes $F \colon G\to G^{(1)}$. The \emph{$r$th Frobenius kernel} of $G$ is the group scheme theoretic kernel of the $r$-fold iteration of the Frobenius map:
 \[ 
 G_{(r)} =\Ker (F^r \colon G\to G^{(r)}).
 \] 
The Frobenius kernel of $G$ is connected if the $k$-algebra $k[G]$ is finitely generated.

Let $\bbG_a$ denote the additive group over $k$. For the $r$-th Frobenius kernel $\bbG_{a(r)}$ there are isomorphism of $k$-algebras
\[ 
k[\bbG_{a(r)}]\cong k[t]/(t^{p^r}) \quad\text{and}\quad
k\bbG_{a(r)}\cong k[u_0,\dots,u_{r-1}]/(u_0^p,\dots,u_{r-1}^p).
\]
\end{example}

\begin{example}[Quasi-elementary group schemes]
  \label{ex:quasi-elementary} 
 Following Bendel~\cite{Bendel:2001a}, a group scheme is said to be \emph{quasi-elementary} if it is   isomorphic to $\bbG_{a(r)} \times (\bbZ/p)^s$. Thus a quasi-elementary group scheme is unipotent abelian and its group algebra structure is the same as that of an elementary abelian $p$-group of rank $r+s$. Note that any \emph{connected} quasi-elementary group scheme is isomorphic to $\bbG_{a(r)}$, for some $r$.
\end{example}

A module over an affine group scheme $G$ over $k$ is called \emph{$G$-module}; it is equivalent to a comodule over the Hopf algebra $k[G]$. The term ``module'' will mean ``left module'', unless stated otherwise. We write $\Mod G$ for the category of $G$-modules and $\mod G$ for its full subcategory consisting of finite dimensional $G$-modules. When $G$ is finite, we identify $G$-modules with modules over the group algebra $kG$; this is justified by \cite[I.8.6]{Jantzen:2003a}.  

\subsection*{Induction} For each subgroup scheme $H$ of $G$ restriction is a functor
\[ 
\res^G_H\colon \Mod G \lra \Mod H.
\] 
We often write $(-)\da_{H}$ instead of $\res^{G}_{H}$. This has a right adjoint called induction\footnote{Warning: in representation theory of finite groups, \emph {induction} is commonly used for the \emph{left adjoint}. We stick with the convention in  \cite{Jantzen:2003a} pointing out that for group schemes the left adjoint does not always exist and when it does, it is not necessarily isomorphic to the right adjoint.}
\[ 
\ind^G_H\colon \Mod H \lra \Mod G
\] 
as described in \cite[I.3.3]{Jantzen:2003a}. If the quotient $G/H$ is affine then $\ind^G_H$ is exact. This holds, for example, when $H$ is
finite; see \cite[I.5.13]{Jantzen:2003a}.

\subsection*{Extending the base field} 
Let $G$ be a group scheme over $k$.  If $K$ is a field extension of $k$, we write $K[G]$ for $K \otimes_k k[G]$, which is a commutative Hopf algebra over $K$. This defines a group scheme over $K$ denoted $G_K$. If $G$ is a finite group scheme, then there is a natural isomorphism $KG_K\cong K\otimes_k kG$ and we simplify notation by writing $KG$.  For a $G$-module $M$, we set
\[ 
M_K:=K \otimes_k M\,. 
\]
The induction functor commutes with the extension of scalars (see \cite[I.3.5]{Jantzen:2003a}), that is, there is a canonical
isomorphism:
\[ 
\ind^{G_K}_{H_K}(M_K) \cong (\ind^G_H M)_K\,.
\] 
The assignment $M\mapsto M_{K}$ defines a functor from $\Mod G$ to $\Mod G_{K}$ which is left adjoint to the restriction functor $ \Mod
G_K \to \Mod G$.

For $G$ a \emph{finite} group scheme and a $G$-module $M$ we set
\[ 
M^K := \Hom_k(K,M),
\] 
again viewed as a $G_K$-module. This is right adjoint to restriction.  It is essential for the group scheme to be finite to make sense of this definition; see Remark~\ref{re:cosupp}.

\subsection*{$\pi$-points} 
Let $G$ be a finite group scheme over $k$.  A $\pi$-\emph{point} of $G$, defined over a field extension $K$ of $k$,
is a morphism of $K$-algebras
\[ 
\alpha\colon K[t]/(t^p) \lra KG
\] 
that factors through the group algebra of a unipotent abelian subgroup scheme $C$ of $G_{K}$, and such that $KG$ is flat when viewed
as a left (equivalently, as a right) module over $K[t]/(t^{p})$ via $\alpha$. It should be emphasised that $C$ need not be defined over
$k$; see Example~\ref{ex:sl2}. Restriction along $\alpha$ defines a
functor
\begin{align*} 
\alpha^{*}&\colon \Mod G_K \lra \Mod K[t]/(t^{p}),\\
\intertext{and the functor $KG\otimes_{K[t]/(t^{p})} -$ provides a left adjoint} \alpha_* & \colon \Mod
K[t]/(t^{p}) \lra \Mod G_K.
\end{align*}

\begin{definition}
\label{de:pi} 
A pair of $\pi$-points $\alpha\colon K[t]/(t^p)\to KG$ and $\beta\colon L[t]/(t^p)\to LG$ are \emph{equivalent}, denoted $\alpha\sim\beta$, if they satisfy the following condition: for any finite dimensional $kG$-module $M$, the module $\alpha^*(M_K)$ is projective if and only if $\beta^*(M_L)$ is projective (see \cite[Section 2]{Friedlander/Pevtsova:2007a} for a discussion of the equivalence relation).
\end{definition}

\begin{remark}
\label{re:pi-basics} 
For ease of reference, we list some basic properties of $\pi$-points.

(1) Let $\alpha\colon K[t]/(t^p)\to KG$ be a $\pi$-point and $L$ a
field extension of $K$. Then $L\otimes_{K}\alpha\colon L[t]/(t^p)\to
LG$ is a $\pi$-point and it is easy to verify that $\alpha\sim L
\otimes_K \alpha$.

(2) Every $\pi$-point is equivalent to one that factors through some
quasi-elementary subgroup scheme.  This is proved in
\cite[Proposition~4.2]{Friedlander/Pevtsova:2005a}.

(3) Every $\pi$-point of a subgroup scheme $H$ of $G$ is naturally a
$\pi$-point of $G$.

(4) A $\pi$-point $\alpha$ of $G$ defined over $L$ naturally gives
rise to a $\pi$-point of $G_K$ defined over a field containing $K$ and
$L$.
\end{remark}

For quasi-elementary group schemes, one has a concrete description of $\pi$-points and the equivalence relation between them.

\begin{example}
\label{ex:pi-point}
Let $\mcE$ be a quasi-elementary group scheme defined over $k$ of positive characteristic $p$. For any field extension $K$ of $k$, the group algebra $K\mcE$ is isomorphic to the $K$-algebra $K[z_{1},\dots,z_{n}]/(z_{1}^{p},\dots, z_{n}^{p})$; see Example~\ref{ex:quasi-elementary}. Since $\mcE$ is unipotent abelian, a $\pi$-point defined over $K$ is nothing but a flat map of $K$-algebras
\[
\alpha\colon K[t]/(t^{p})\lra K[z_{1},\dots,z_{n}]/(z_{1}^{p},\dots, z_{n}^{p}).
\]
What is more, flatness of $\alpha$ is equivalent to the condition that $\alpha(t)$ has a linear part; see \cite[Proposition~2.2]{Friedlander/Pevtsova:2005a}. The same result also yields that $\pi$-points 
\[
\alpha,\beta\colon K[t]/(t^{p})\lra K[z_{1},\dots,z_{n}]/(z_{1}^{p},\dots, z_{n}^{p})
\]
are equivalent if and only if $\alpha(t)$ and $\beta(t)$ have proportional linear parts.
\end{example}

\subsection*{$\pi$-points and cohomology}
Let $G$ be a finite group scheme over $k$.  The cohomology of $G$ with
coefficients in a $G$-module $M$ is denoted $H^{*}(G,M)$. It can be
identified with $\Ext_{G}^{*}(k,M)$, with the trivial action of $G$ on
$k$.  Recall that $H^{*}(G,k)$ is a $k$-algebra that is
graded-commutative (because $kG$ is a Hopf algebra) and finitely
generated, by a theorem of
Friedlander--Suslin~\cite[Theorem~1.1]{Friedlander/Suslin:1997a}. 

Let $\Proj H^*(G, k)$ denote the set of homogeneous prime ideals $H^*(G, k)$ that are properly contained in the maximal ideal of
positive degree elements.

Given a $\pi$-point $\alpha\colon K[t]/(t^{p})\to KG$ we write $H^{*}(\alpha)$ for the composition of homomorphisms of $k$-algebras.
\[ 
H^{*}(G,k) = \Ext^{*}_{G}(k,k) \xra{\ K\otimes_{k}-} \Ext^{*}_{G_K}(K,K) \lra \Ext^{*}_{K[t]/(t^{p})}(K,K),
\] 
where the one on the right is induced by restriction. Evidently, the radical of the ideal $\Ker H^{*}(\alpha)$ is a prime ideal in
$H^*(G,k)$ different from $H^{\ges 1}(G,k)$ and so defines a point in $\Proj H^{*}(G,k)$.

\begin{remark}
\label{re:generic-points} Fix a point $\fp$ in $\Proj H^{*}(G,k)$. There exists a field $K$ and a $\pi$-point
\[ 
\alpha_\fp\colon K[t]/(t^p)\lra KG
\] 
such that $\sqrt{\Ker H^{*}(\alpha_{\fp})}=\fp$. In fact, there is such a $K$ that is a finite extension of the degree zero part of the
homogenous residue field at $\fp$; see \cite[Theorem~4.2]{Friedlander/Pevtsova:2007a}. It should be emphasised that $\alpha_{\fp}$ is not uniquely defined.
\end{remark}

In this way, one gets a bijection between the set of equivalence classes of $\pi$-points of $G$ and $\Proj H^*(G,k)$; see \cite[Theorem~3.6]{Friedlander/Pevtsova:2007a}.  In the sequel, we identify a prime in $\Proj H^{*}(G,k)$ and the corresponding equivalence class of $\pi$-points.

The following definitions of $\pi$-support and $\pi$-cosupport of a $G$-module $M$ are from \cite{Friedlander/Pevtsova:2007a} and
\cite{Benson/Iyengar/Krause/Pevtsova:2015a} respectively.

\begin{definition}
\label{de:cosupport} Let $G$ be a finite group scheme and $M$ be a
$G$-module.  The \emph{$\pi$-support} of $M$ is the subset of $\Proj
H^{*}(G,k)$ defined by
\[ 
\pisupp_{G}(M) := \{\fp\in\Proj H^*(G,k) \mid \text{$\alpha_\fp^*(K\otimes_k M)$ is not projective}\}.
\] 
The \emph{$\pi$-cosupport} of $M$ is the subset of $\Proj H^{*}(G,k)$ defined by
\[ 
\picosupp_{G}(M) := \{\fp\in\Proj H^*(G,k) \mid \text{$\alpha_\fp^*(\Hom_k(K,M))$ is not projective}\}.
\] 
Here $\alpha_\fp\colon K[t]/(t^p)\to KG$ denotes a representative of the equivalence class of $\pi$-points corresponding to $\fp$; see
Remark~\ref{re:generic-points}. Both $\pisupp$ and $\picosupp$ are well defined on the equivalence classes of $\pi$-points; see
\cite[Theorem 3.1]{Benson/Iyengar/Krause/Pevtsova:2015a}.
\end{definition}

The following observation will be useful; see Corollary~\ref{co:maxsupp=maxcosupp} for a better statement.

\begin{lemma}
\label{le:maxsupp=maxcosupp} 
Let $\fm$ be a closed point of $\Proj H^{*}(G,k)$ and $M$ a $G$-module. Then $\fm$ is in $\pisupp_{G}(M)$ if and only if it is in $\picosupp_{G}(M)$.
\end{lemma}

\begin{proof} 
The key observation is that as $\fm$ is a closed point there is a corresponding $\pi$-point $\alpha_{\fm}\colon K[t]/(t^p)\to KG$ with $K$ a finite extension of $k$; see Remark~\ref{re:generic-points}. It then remains to note that the natural evaluation map is an isomorphism:
\[ 
\Hom_{k}(K,k)\otimes_k M \xra{\ \cong\ }\Hom_{k}(K,M)
\] 
so that $K\otimes_{k}M$ and $\Hom_{k}(K,M)$ are isomorphic as $G_{K}$-modules.
\end{proof}

The next result plays a key role in what follows. The formula for the support of tensor products is \cite[Proposition~5.2]{Friedlander/Pevtsova:2007a}.  The proof of the formula for the cosupport of function objects is similar, and is given in \cite[Theorem 4.4]{Benson/Iyengar/Krause/Pevtsova:2015a}. We sketch it here for the reader's convenience.

\begin{theorem}
\label{th:tensor-and-hom-pi} 
Let $M$ and $N$ be $G$-modules. Then there are equalities
\begin{enumerate}[{\quad\rm(1)}]
\item $\pisupp_{G}(M \otimes_k N) = \pisupp_{G}(M) \cap \pisupp_{G}(N)$,
\item $\picosupp_{G}(\Hom_k(M,N)) = \pisupp_{G}(M) \cap \picosupp_{G}(N)$. 
\end{enumerate}
\end{theorem}

\begin{proof}   
Remark~\ref{re:pi-basics}(2) implies that we can assume that $G$ is a quasi-elementary 
group scheme.  Hence, $kG$ is isomorphic to $k[t_{1},\dots,t_{r}]/(t_{1}^{p},\dots t_{r}^{p})$ as an algebra.  

Let $\alpha \colon K[t]/(t^{p})\to KG$ be a  $\pi$-point of $G$. Extending scalars and using 
the isomorphism $\alpha^*(\Hom_k(M,N)^K)\cong \alpha^*(\Hom_K(M_K,N^K))$, we may 
assume that $\alpha$ is defined over $k$. To prove the equality (2), we need to show that 
$\alpha^*(\Hom_k(M,N))$ is projective if and only if $\alpha^*(M)$ or $\alpha^*(N)$  is projective. 

Let $\sigma\colon kG\to kG$ be the antipode of $kG$, $\Delta\colon kG\to kG\otimes kG$ 
its comultiplication, and set $I=\Ker(kG\to k)$, the augmentation ideal of $kG$.  
Identifying $t$ with its image in $kG$ under $\alpha$, one has
\[ 
(1\otimes\sigma)\Delta(t) = t\otimes 1 - 1\otimes t + w \quad\text{with $w\in I\otimes I$;}
\] 
see \cite[I.2.4]{Jantzen:2003a}. 

Given a module over $kG\otimes kG$, we consider two $k[t]/(t^{p})$-structures on it: 
One where $t$ acts via multiplication with $(1\otimes\sigma)\Delta(t)$ and another where 
it acts via multiplication with $t\otimes 1 - 1\otimes t$. We claim that these two $k[t]/(t^{p})$-modules 
are both projective or both not projective. This follows from a repeated use of 
\cite[Proposition~2.2]{Friedlander/Pevtsova:2005a} because $w$ can be represented as 
a sum of products of nilpotent elements of $kG\otimes kG$, where each nilpotent element 
$x$ satisfies $x^p=0$.

We may thus assume that $t$ acts on $\Hom(M,N)$ via $t\otimes 1 - 1\otimes t$. 
There is then an isomorphism of $k[t]/(t^{p})$-modules
\[ 
\alpha^{*}(\Hom_k(M,N)) \cong \Hom_k(\alpha^{*}(M),\alpha^{*}(N))\,,
\] 
where the action of $k[t]/(t^{p})$ on the right hand side is the one obtained by viewing 
it as a Hopf algebra with comultiplication defined by $t\mapsto t\otimes 1 + 1\otimes t$ 
and antipode $t\mapsto -t$. It remains to observe that for any $k[t]/(t^{p})$-modules $U,V$,
the module $\Hom_k (U,V)$ is projective if and only if one of $U$ or $V$ is projective.
\end{proof} 

\section{Support and cosupport via cohomology}
\label{se:bik} 
This section provides a quick summary of the techniques developed in \cite{Benson/Iyengar/Krause:2008a,Benson/Iyengar/Krause:2012b}, with a focus on  modules over finite group schemes. Throughout $G$ will be a finite group scheme over a field $k$.

\subsection*{The stable module category} 
The stable module category of $G$ is denoted $\StMod G$. Its objects are all (possibly infinite dimensional) $G$-modules. The set of morphisms between $G$-modules $M$ and $N$ is by definition
\[ 
\sHom_{G}(M,N) := \frac{\Hom_G(M,N)}{\mathrm{PHom}_{G}(M,N)}
\] 
where ${\rm PHom}_G(M,N)$ consists of all $G$-maps between $M$ and $N$ which factor through a projective $G$-module.  Since $G$-modules are precisely modules over the group algebra $kG$ and the latter is a Frobenius algebra~\cite[Lemma~I.8.7]{Jantzen:2003a}, the stable module category is triangulated, with suspension equal to $\Omega^{-1}(-)$, the inverse of the syzygy functor. The tensor product $M\otimes_{k}N$ of $G$-modules, with the usual diagonal $G$-action, is inherited by $\StMod G$ making it a tensor triangulated category. This category is compactly generated and the subcategory of compact objects is equivalent to $\stmod G$, the stable module category of finite dimensional $G$-modules. For details, readers might consult Carlson \cite[\S5]{Carlson:1996a} and Happel~\cite[Chapter 1]{Happel:1988a}.
 
A subcategory of $\StMod G$ is said to be \emph{thick} if it is a triangulated subcategory that is closed under direct summands. 
Note that any triangulated subcategory is closed under finite direct sums. A triangulated subcategory that is closed under all set-indexed
direct sums is said to be \emph{localising}. A localising subcategory is also thick. We say that a subcategory  is a \emph{tensor ideal} if it is
closed under $M\otimes_{k}-$ for any $G$-module $M$.
 
We write $\sHom_{G}^{*}(M,N)$ for the graded abelian group with $\sHom_{G}(M,\Omega^{-i}N)$ the component in degree $i$. Composition of morphisms endows $\sHom^{*}_{G}(M,M)$ with a structure of a graded ring and $\sHom_{G}^{*}(M,N)$ with the structure of a graded left-$\sHom^{*}_{G}(N,N)$ and right-$\sHom^{*}_{G}(M,M)$ bimodule. There is a natural map
\[ 
\Ext^{*}_{G}(M,N)\lra \sHom_{G}^{*}(M,N)
\] 
of graded abelian groups; it is a homomorphism of graded rings for $M=N$. Composing this with the homomorphism
\[ 
-\otimes_{k}M \colon H^{*}(G,k)\lra \Ext^{*}_{G}(M,M)
\] 
yields a homomorphism of graded rings
\[ 
\phi_{M}\colon H^{*}(G,k)\lra \sHom_{G}^{*}(M,M)\,.
\] 
It is clear from the construction that $\sHom_{G}^{*}(M,N)$ is a graded bimodule over $H^{*}(G,k)$ with left action induced by
$\phi_{N}$ and right action induced by $\phi_{M}$, and that the actions differ by the usual sign. Said otherwise, $H^{*}(G,k)$ \emph{acts} on $\StMod G$, in the sense of \cite[Section~3]{Benson/Iyengar/Krause:2012b}.

\subsection*{The spectrum of the cohomology ring} 
We write $\Spec H^{*}(G,k)$ for the set of homogenous prime ideals in $H^{*}(G,k)$. It has one more point than $\Proj H^{*}(G,k)$, namely, the maximal ideal consisting of elements of positive degree. A subset $\mcV$ of $\Spec H^{*}(G,k)$ is \emph{specialisation closed} if whenever $\fp$ is in $\mcV$ so is any prime $\fq$ containing $\fp$. Thus $\mcV$ is specialisation closed if and only if it is a union of Zariski closed subsets, where a subset is \emph{Zariski closed} if it is of the form
\[ 
\mcV(I):=\{\fp\in \Spec H^{*}(G,k)\mid \fp\subseteq I\}
\] 
for some ideal $I$ in $H^{*}(G,k)$. In the sequel, given  $\fp\in\Spec H^{*}(G,k)$ and  a graded $H^{*}(G,k)$-module $N$, we write $N_{\fp}
$ for the homogeneous localisation of $N$ at $\fq$.

\subsection*{Local cohomology} 
Let $\mcV$ be a specialisation closed subset of $\Spec H^*(G,k)$. A $G$-module $M$ is called \emph{$\mcV$-torsion} if $\sHom^{*}_{G}(C,M)_{\fq} =0$ for any finite dimensional $G$-module $C$ and $\fq\not\in \mcV$. We write  $ (\StMod G)_{\mcV}$ for the full subcategory of $\mcV$-torsion modules. This is a localising subcategory and the inclusion $(\StMod G)_{\mcV}\subseteq \StMod G$ admits a right adjoint, denoted $\gam_{\mcV}$. Thus, for each $M$ in $\StMod G$ one gets a functorial
exact triangle
\[ 
\gam_{\mcV}M \lra M \lra L_{\mcV}M\lra
\] 
and this provides a localisation functor $L_\mcV$ that annihilates precisely the $\mcV$-torsion modules. For details, see
\cite[Section~4]{Benson/Iyengar/Krause:2008a}.

A noteworthy special case pertains to a point $\fp$ in $\Proj H^{*}(G,k)$ and the subset
\[ 
\mcZ(\fp):=\{\fq\in \Spec H^{*}(G,k)\mid \fq \not\subseteq \fp\}.
\] 
This is evidently a specialisation closed subset. The corresponding localisation functor $L_{\mcZ(\fp)}$ models localisation at $\fp$,
that is to say, for any $G$-module $M$ and finite dimensional $G$-module $C$, the morphism $M\to L_{\mcZ(\fp)}M$ induces an
isomorphism
\[ 
\sHom^{*}_{G}(C,M)_\fp\stackrel{\cong}\lra \sHom^{*}_{G}(C,L_{\mcZ(\fp)}M)
\]
of graded $H^{*}(G,k)$-modules; see \cite[Theorem~4.7]{\bik:2008a}. For this reason, we usually write $M_{\fp}$ in lieu of $L_{\mcZ(\fp)}M$. When the natural map $M\to M_{\fp}$ is an isomorphism we say $M$ is \emph{$\fp$-local}, and that $M$ is \emph{$\fp$-torsion} if it is $\mcV(\fp)$-torsion.

We write $\gam_{\fp}$ for the exact functor on $\StMod G$ defined on objects by
\[ 
\gam_{\fp}M:= \gam_{\mcV(\fp)}(M_{\fp}) = (\gam_{\mcV(\fp)} M)_{\fp}\,.
\] 
The equality is a special case of a general phenomenon: the functors $\gam_{\mcV}$ and $L_{\mcW}$ commute for any specialisation
closed subsets $\mcV$ and $\mcW$; see \cite[Proposition~6.1]{\bik:2008a}. This property will be used often and without further comment.

\subsection*{Support and cosupport} 
We introduce the \emph{support} of a $G$-module $M$ to be the following subset of $\Proj H^{*}(G,k)$.
\[ 
\supp_{G}(M):=\{\fp \in \Proj H^{*}(G,k)\mid \gam_{\fp}M \text{ is not projective}\}
\] 
As in \cite[Section~4]{Benson/Iyengar/Krause:2012b}, the \emph{cosupport} of $M$ is the set
\[ 
\cosupp_{G}(M) := \{\fp\in\Proj H^*(G,k) \mid \text{$\Hom_k(\gam_\fp k,M)$ is not projective}\}.
\] 
Note that we are ignoring the closed point of $\Spec H^{*}(G,k)$. It turns out that the support and the cosupport of $M$ coincide with its $\pi$-support and $\pi$-cosupport introduced in Section~\ref{se:pi-points}; see Theorems~\ref{th:pisupp=bik} and \ref{th:picosupp=bik}.

\subsection*{Stratification} 
Let $(\sfT,\otimes,\one)$ be a compactly generated tensor triangulated category and $R$ a graded-commutative noetherian ring acting on $\sfT$ via homomorphisms of rings $R\to \End^{*}(X)$, for each $X$ in $\sfT$; see \cite[Section~8]{\bik:2008a} for details. For each $\fp$ in $\Spec R$, one can construct a functor $\gam_{\fp}\colon \sfT\to \sfT$ as above and use it to define support and cosupport for objects in $\sfT$. The subcategory
\[ 
\gam_{\fp}\sfT:=\{ X\in \sfT\mid X\cong\gam_{\fp}X\}
\]
consists of all objects $X$ in $\sfT$ such that $\Hom_\sfT^*(C,X)$ is $\fp$-local and $\fp$-torsion for each compact object $C$, and has the following alternative description:
\[ 
\gam_{\fp}\sfT=\{ X\in \sfT\mid \supp_R(X)\subseteq\{\fp\}\};
\]
see \cite[Corollary~5.9]{\bik:2008a}. The subcategory $\gam_{\fp}\sfT$ of $\sfT$ is tensor ideal and localising. 

We say that $\gam_{\fp}\sfT$ is \emph{minimal} if it is non-zero and contains no proper non-zero tensor ideal localising subcategories. Following \cite[Section~7]{\bik:2011a} we say $\sfT$ is \emph{stratified} by $R$ if for each $\fp$ the subcategory $\gam_{\fp}\sfT$ is either zero or minimal. When this property holds, the tensor ideal localising subcategories of $\sfT$ are parameterised by subsets of $\Spec R$; see \cite[Theorem~3.8]{\bik:2011b}. The import of this statement is that the classification problem we have set out to solve can be tackled one prime at a time.

Lastly, we recall from \cite[Section~7]{\bik:2012b} the behaviour of support under change of rings and categories.

\subsection*{Change of rings and categories} In this paragraph,
$(\sfT,R)$ denotes a pair consisting of a compactly generated
triangulated category $\sfT$ endowed with an action of a
graded-commutative noetherian ring $R$. A functor $(F,\phi)\colon
(\sfT,R)\to (\sfU,S)$ between such pairs consists of an exact functor
$F\colon\sfT\to\sfU$ that preserves set-indexed products and
coproducts, and a homomorphism $f\colon R\to S$ of rings such that,
for each $X\in \sfT$, the following diagram is commutative.
\[
\begin{tikzcd} R\arrow{r}{f} \arrow{d} & S\arrow{d} \\
\End_{\sfT}^*(X) \arrow{r}{F} & \End_{\sfU}^*(FX)
\end{tikzcd}
\] The result below is extracted from
\cite[Corollary~7.8]{\bik:2012b}.

\begin{proposition}
\label{pr:change-cat-ring} 
Let $(F,f)\colon (\sfT,R) \to (\sfU,S)$ be a functor between compactly generated triangulated categories with ring actions. Let $E$ be a left adjoint of $F$, let $G$ be a right adjoint of $F$, and $\phi\colon\Spec S\to\Spec R$ the map that assigns $f^{-1}(\fp)$ to $\fp$. Then for $X$ in $\sfT$ and $Y$ in $\sfU$ there are inclusions:
\begin{enumerate}[\quad\rm(1)]
\item $\phi(\supp_{S}(FX))\subseteq \supp_{R}(X)\quad\text{and}\quad
\supp_{R}(EY)\subseteq\phi(\supp_{S}(Y))$,
\item $\phi(\cosupp_{S}(FX))\subseteq \cosupp_{R}(X)\quad\text{and}\quad
\cosupp_{R}(GY)\subseteq\phi(\cosupp_{S}(Y)).$
\end{enumerate} Each inclusion is an equality if the corresponding functor is
faithful on objects.\qed
\end{proposition}

\part{Detecting projectivity with $\pi$-support}
\label{part:detection}

Let $G$ be a finite group scheme over a field of positive characteristic. This part is dedicated to a proof of Theorem~\ref{th:main} that
asserts that \emph{$\pi$-support detects projectivity} of $G$-modules, by which we mean that a $G$-module $M$ is projective if
(and only if) $\pisupp_{G}(M)=\varnothing$. This result was claimed in \cite{Friedlander/Pevtsova:2007a}, but the argument there has a flaw (see Remark~\ref{re:fpnot}) which we repair here. Most of the different pieces of our proof are already available in the literature; we
collect them here for the sake of both completeness and comprehensibility. 

The essential new ingredient is a ``subgroup  reduction principle'', Theorem~\ref{th:principle} which allows us  to extend the detection theorem from the case of a connected finite group  scheme to an arbitrary one. Theorem~\ref{th:principle} relies on a remarkable result of Suslin (see also \cite{Bendel:2001a} for the special case of a unipotent finite group scheme) on detection of nilpotents in the cohomology ring $H^*(G,\Lambda)$ for a $G$-algebra $\Lambda$, generalising work of Quillen and Venkov  for finite groups.

The first step in our proof of the detection theorem is to settle the case of a \emph{connected unipotent} finite group scheme. This is achieved in Section~\ref{se:unipotent}.  The argument essentially follows the one of Bendel \cite{Bendel:2001a} but is simpler, for two reasons: because of the connectedness assumption and because we employ the subgroup reduction principle that allows one to apply induction on $\dim_{k} k[G]$ in certain cases. 

The subgroup reduction principle cannot be used for general connected finite groups schemes; see Example~\ref{ex:sl2}. To tackle that case, we import a result from \cite{Pevtsova:2002a} which readily implies that $\pi$-support detects projectivity for Frobenius kernels of connected reductive groups; in fact it would suffice to treat $\GL_{n(r)}$, but the proof is no different in general.  A connected
group scheme can be realised as a subgroup of a Frobenius kernel and so we deduce the desired property for the former from that for the
latter using a descent theorem. This is done in Section~\ref{se:generalcase} and essentially repeats the argument in~\cite{Pevtsova:2004a}.
This also takes care of group schemes that are a direct product of their identity component with an elementary abelian $p$-group. After all, the statement of the theorem does not mention the coalgebra part of the structure, and in this case the algebra structure is identical to that of a suitably chosen connected finite group scheme.

Armed with these results, we tackle the general case, also in Section~\ref{se:generalcase}, but not without yet another invocation
of the subgroup reduction principle, Theorem~\ref{th:principle}.

\section{A subgroup reduction principle}
\label{se:an-induction-principle} 
In this section we establish basic results, including the general subgroup reduction principle alluded to above, Theorem~\ref{th:principle}, that will be used repeatedly in proving that $\pi$-support detects projectivity.  Throughout, $G$ will
be a finite group scheme over a field $k$ of positive characteristic.

\begin{lemma}
\label{le:vanish} 
Let $G$ be a finite group scheme with the property that for any $G$-module $M$ with $\pisupp_G(M) = \varnothing$ one has
$H^i(G,M)=0$ for $i\gg 0$. Then $\pi$-support detects projectivity of $G$-modules.
\end{lemma}

\begin{proof} 
Let $M$ be a $G$-module with $\pisupp_{G}(M)=\varnothing$. Then, for any simple $G$-module $S$, Theorem~\ref{th:tensor-and-hom-pi} yields
\[ 
\pisupp_{G}(\Hom_{k}(S,k) \otimes_{k} M)  =
\pisupp_{G}(\Hom_{k}(S,k)) \cap \pisupp_{G}(M)  =\varnothing\,.
\]
Thus, for $i\gg 0$ the hypothesis on $G$ gives the second equality below:
\[ 
\Ext_G^{i}(S, M)\cong H^{i}(G, \Hom_{k}(S,k) \otimes_{k} M) =0\,,
\] 
where the isomorphism holds since all simple $G$-modules are finite dimensional. It follows that $M$ is projective, as desired.
\end{proof}

The following observation will be of some use in what follows.

\begin{remark}
\label{re:groupalgebras}
If $G$ and $G'$ are unipotent abelian group schemes such that their group algebras are isomorphic, then $\pi$-support detects projectivity of $G$-modules if and only if it detects projectivity of $G'$-modules.

Indeed, this is because projectivity of a $G$-module $M$ does not involve the comultiplication on $kG$, and when $G$ is unipotent abelian $\pi$-points are just flat homomorphism of $K$-algebras $K[t]/(t^{p})\to KG$, for some field extension $K/k$, and again have nothing to do with the comultiplication on $KG$.
\end{remark}

To establish that $\pi$-support detects projectivity we need a version of Dade's lemma proved by Benson, Carlson, and Rickard in
\cite[Lemma~4.1]{Benson/Carlson/Rickard:1996a}. For our purposes we restate the result in terms of $\pi$-support as can be found in
\cite[Theorem 5.4]{Benson/Iyengar/Krause/Pevtsova:2015a}.

\begin{theorem}
\label{thm:Dade} 
If $\mcE$ is a quasi-elementary group scheme, then $\pi$-support detects projectivity of $\mcE$-modules.
\end{theorem}

\begin{proof} 
The group algebra $k\mcE$ of a quasi-elementary group scheme is isomorphic to the group algebra of an elementary abelian
finite group as seen in Example~\ref{ex:quasi-elementary}. In view of Remark~\ref{re:groupalgebras},  the result follows from \cite[Theorem 5.4]{Benson/Iyengar/Krause/Pevtsova:2015a}.
 \end{proof}

The next result, which is a corollary of Suslin's theorem on detection of nilpotence in cohomology \cite[Theorem~5.1]{Suslin:2006a}, is critical to what follows.

\begin{theorem}
\label{th:suslin} 
Let $G$ be a finite group scheme over a field $k$ and $\Lambda$ a unital associative $G$-algebra. If $\pisupp_{G}(\Lambda) =\varnothing$, then any element in $H^{\ges 1}(G, \Lambda)$ is nilpotent.
\end{theorem}

\begin{proof} 
For any extension field $K$ of $k$ and any quasi-elementary subgroup scheme $\mcE$ of $G_K$, the hypothesis of the theorem yields
\[ 
\pisupp_{\mcE} (\Lambda_K)\da_{\mcE} = \varnothing\,.
\] 
Theorem~\ref{thm:Dade} then yields that $ (\Lambda_K)\da_{\mcE}$ is projective, so $H^{\ges 1}(\mcE, \Lambda_K) =0$. This implies that for any element $z \in H^{\ges 1}(G, \Lambda)$ the restriction of $z_K$ to $H^*(\mcE,  (\Lambda_K)\da_{\mcE})$ is trivial. Therefore, $z$ is nilpotent, by \cite[Theorem~5.1]{Suslin:2006a}.
\end{proof}

The next result has been proved in a larger context by Burke~\cite[Theorem]{Burke:2012a}.  For finite group schemes a simpler
argument is available and is given below.

\begin{lemma}
\label{le:nilp} 
Let $G$ be a finite group scheme and $M$ a $G$-module. If each element in $\Ext_G^{\ges 1}(M,M)$ is nilpotent,
then $M$ is projective.
\end{lemma}

\begin{proof} 
  The $k$-algebra $H^*(G,k)$ is finitely generated so Noether 
  normalisation provides homogeneous algebraically independent
  elements $z_1, \ldots, z_r$ in $H^{\ges 1}(G,k)$ such that the
  extension $k[z_1, \ldots, z_r] \subseteq H^*(G,k)$ is finite;
  see~\cite[Theorem~1.5.17]{Bruns/Herzog:1998a}. By assumption, the
  image of any $z_i$ under the composition
\[
k[z_1, \ldots, z_r] \lra H^*(G,k) \lra \Ext_G^*(M,M)
\] 
is nilpotent. By taking powers of the generators $z_1, \ldots, z_r$, if necessary, one may assume that these images are zero.

Represent each $z_{i}$ by a homomorphism $\Omega^{|z_{i}|}(k)\to k$ of
$G$-modules and let $L_{z_{i}}$ denote its kernel; this is the
Carlson module~\cite{Carlson:1983a} associated to $z_i$. Vanishing of
$z_1$ in $\Ext_G^*(M,M)$ implies that for some integer $n$, the
$G$-modules $L_{z_1} \otimes_{k} M$ and $\Omega M \oplus \Omega^n M$
are isomorphic up to projective summands; this is proved in
\cite[5.9]{Benson:1998c} for finite groups and the argument carries over
to finite group schemes. Setting $L_{\bsz}:= L_{z_r} \otimes_{k}
\cdots \otimes_{k} L_{z_1}$, an iteration yields that the $G$-modules
\[ 
L_{\bsz} \otimes_k M \quad \text{and}\quad \bigoplus_{i=1}^{2^r}
\Omega^{n_i}M
\] are isomorphic up to projective summands. However, since $H^*(G,k)$
is finite as a module over $k[z_1, \ldots, z_r]$, one gets the second
equality below
\[ 
\pisupp_G(L_{\bsz}) = \bigcap\limits_{i=1}^r \pisupp_G(L_{z_i}) =
\varnothing.
\] 
The first one is by Theorem~\ref{th:tensor-and-hom-pi}. As the
$G$-module $L_{\bsz}$ is finitely generated, by construction, it
follows that $L_{\bsz}$, and hence also $L_{\bsz} \otimes_{k} M$, is
projective; see, for example,
\cite[Theorem~5.6]{Friedlander/Pevtsova:2005a}.  Thus
$\bigoplus_{i=1}^{2^r} \Omega^{n_i}M$ is projective, and hence so is
$M$.
\end{proof}

The next result is well-known; for a proof see, for example,
\cite[Lemma~4.2]{Benson/Iyengar/Krause/Pevtsova:2015a}.

\begin{lemma}
\label{le:end-projectivity} Let $M$ and $N$ be $G$-modules.
\begin{enumerate}[{\quad\rm(1)}]
\item If $M$ or $N$ is projective, then so is $\Hom_{k}(M,N)$.
\item $M$ is projective if and only if $\End_{k}(M)$ is
projective. \qed
\end{enumerate}
\end{lemma}

We can now establish the following subgroup reduction principle.

\begin{theorem}
  \label{th:principle} 
Let $G$ be a finite group scheme over $k$ with the property that every $\pi$-point for $G$ is equivalent to a $\pi$-point that factors through an embedding $H_K \hookrightarrow G_K$ where $H$ is a proper subgroup scheme of $G$ and $K/k$ is a field extension.  If $\pi$-support detects projectivity for all proper subgroup schemes of $G$, then it detects projectivity for $G$.
\end{theorem}

We emphasise that $H$ is already defined over $k$.

\begin{proof} 
Let $M$ be a $G$-module with $\pisupp_G(M) = \varnothing$.  Let $H$ be a proper subgroup scheme of $G$.  Any
$\pi$-point of $H$ is a $\pi$-point of $G$ so $\pisupp_H(M{\da_H}) =
\varnothing$ and hence $M{\da_H}$ is projective, by
hypothesis. Therefore $\End_{k}(M){\da_H}$ is also projective, by
Lemma~\ref{le:end-projectivity}.  Since any $\pi$-point of $G$ factors
through a proper subgroup scheme, again by hypothesis, one gets that
$\pisupp_G(\End_{k}(M)) =\varnothing$.  By Theorem~\ref{th:suslin},
any element in $H^*(G, \End_{k}(M)) = \Ext_G^*(M,M)$ of positive
degree is nilpotent. Lemma~\ref{le:nilp} then implies that $M$ is
projective, as desired.
\end{proof}

The hypothesis of Theorem~\ref{th:principle} is quite restrictive, as the next example shows.

\begin{example}
\label{ex:sl2} 
Let $k$ be a field of characteristic at least $3$ and $\sfg$ the three dimensional Heisenberg Lie algebra over $k$, that is to say, the Lie algebra of $3\times 3$ strictly upper triangular matrices, with zero $p$-power operation.  It has generators $\langle x_1, x_2, x_{12} \rangle$
subject to relations
\[ 
[x_1, x_{12}]=0=[x_2, x_{12}]\quad\text{and}\quad [x_1,x_2]=x_{12}\,.
 \] 
Then $\sfu(\sfg)$, the restricted enveloping algebra of $\sfg$, is a cocommutative Hopf algebra and hence its dual defines a group scheme over $k$. Its support variety is $\bbP^2$ with coordinate algebra $k[y_1, y_2, y_{12}]$.  Let $K = k(y_1, y_2, y_{12})$ be the field of fractions, and let $\alpha: K[t]/(t^p) \to K\otimes_{k} \sfu(\sfg)$ be a ``generic" $\pi$-point given by
\[ 
\alpha_K(t) = y_1x_1 + y_2x_2 + y_{12}x_{12}\,.
\] 
Specialising $\alpha$ to points $[a_1,a_2,a_{12}] \in \bbP^2$ we get all $\pi$-points of $\sfu(\sfg)$ defined over $k$. Therefore $\alpha$ cannot factor through any proper Lie subalgebra of $\sfg$ defined over $k$.
\end{example}

For contexts where Theorem~\ref{th:principle} does apply see
Theorems~\ref{th:unip} and \ref{th:main}.

\section{Connected unipotent group schemes}
\label{se:unipotent}

In this section we prove that $\pi$-support detects projectivity for
modules over connected unipotent finite group schemes.  Our strategy
mimics the one used in \cite{Bendel:2001a}, with one difference: it
does not use the analogue of Serre's cohomological criterion for a
quasi-elementary group scheme as developed in
\cite{Bendel/Friedlander/Suslin:1997b}.  This is because
Theorem~\ref{th:principle} allows us to invoke
\cite[Theorem~1.6]{Bendel/Friedlander/Suslin:1997b} in the step where
Bendel's proof uses Serre's criterion, significantly simplifying the
argument.

\begin{theorem}
  \label{th:unip} 
If $G$ is a connected unipotent finite group scheme over a field $k$, then $\pi$-support detects projectivity.
\end{theorem}

\begin{proof} 
If $G\cong \bbG_{a(r)}$, then a $\pi$-point for $G$ is precisely a flat map of $K$-algebras $K[t]/(t^{p})\to KG$, with $K$ a field extension of $k$. The desired result follows from Theorem~\ref{thm:Dade}, given the description of the group algebra of $\bbG_{a(r)}$
in Example~\ref{ex:frobenius-kernels}.

In the remainder of the proof we may thus assume $G$ is not isomorphic to $\bbG_{a(r)}$. The proof proceeds by induction on $\dim_{k}
k[G]$. The base case, where this dimension is one, is trivial.  Assume that the theorem holds for all proper subgroup schemes of $G$.  We
consider two cases, depending on the rank of $\Hom_{\mathrm{Gr}/k}(G, \bbG_{a(1)})$, the $k$-vector space of morphisms from $G$ to $\bbG_{a(1)}$.

\medskip

\textit{Case}~1.  Suppose $\dim_k\Hom_{\mathrm{Gr}/k}(G, \bbG_{a(1)}) =1$. Let $\phi\colon G \to \bbG_{a(1)}$ be a generator of $\Hom_{\mathrm{Gr}/k}(G, \bbG_{a(1)})$, and $x$ a generator of $H^2(\bbG_{a(1)},k)$.  By \cite[Theorem 1.6]{Bendel/Friedlander/Suslin:1997b}, either $G \cong \bbG_{a(r)}$ or $\phi^*(x) \in H^*(G,k)$ is nilpotent.  Since we have ruled out the former case, we may assume $\phi^*(x)$ is nilpotent.

Let $\alpha\colon K[t]/(t^p) \to KG$ be a $\pi$-point; by Remark~\ref{re:pi-basics}(2) we can assume  it factors through a quasi-elementary subgroup scheme of $G_{K}$. We claim that $\alpha$ is equivalent to a $\pi$-point that factors through $(\Ker \phi)_K$, so that the desired statement follows from Theorem~\ref{th:principle}.

Indeed, consider the composition
\[ 
K[t]/(t^p)\xra{\ \alpha\ } KG \xra{\ \phi_{K}\ } K\bbG_{a(1)}
\] 
and the induced map in cohomology:
\[ 
H^*(\bbG_{a(1)}, K) \xra{\ {\phi_K^*}\ } H^*(G, K) \xra{\ \alpha^*\ } \Ext^{*}_{K[t]/(t^p)}(K,K).
\] 
Since $\phi^*(x) \in H^2(G,k)$ is nilpotent, $(\phi_K \circ \alpha)^*(x_K)=0$. 

The group scheme $G_{K}$ is connected, since $G$ is, so the quasi-elementary subgroup scheme that $\alpha$ factors through must  be isomorphic to $\bbG_{a(r)}$. Restrict $\phi_{K}$ to $\bbG_{a(r)}$ and consider the composition
\[
K[t]/(t^p)\xra{\ \alpha\ } K\bbG_{a(r)} \lra K\bbG_{a(1)}.
\]
It follows from the discussion in the preceding paragraph that the induced map in cohomology is again trivial. Observe that $\Hom_{\mathrm{Gr}/k}(\bbG_{a(r)}, \bbG_{a(1)})$  is one-dimensional and generated by the obvious surjection 
\[
\bbG_{a(r)} \lra \bbG_{a(r)}/\bbG_{a(r-1)} \cong \bbG_{a(1)}.
\]
Thus, if  $K\bbG_{a(r)} = K[u_0, \ldots, u_{r-1}]/(u_i^p)$, this surjection maps each $u_i$ to $0$ for $0\le i\le r-2$ and $u_{r-1}$  to the generator of $K\bbG_{a(1)}$. 

Returning to $\alpha$, we may suppose $\alpha(t) \in K\bbG_{a(r)}$ has a nonzero term with $u_{r-1}$, else it clearly factors through $(\Ker\phi)_{K}$, as desired. The term involved cannot be linear, else the composition $\phi_{K} \circ \alpha$ would be an isomorphism and the induced map in cohomology would not be trivial, which it is. Thus, the terms involving $u_{r-1}$ must be at least quadratic, so $\alpha$ is equivalent to a $\pi$-point with those terms removed; see Example~\ref{ex:pi-point}. That new $\pi$-point then factors through $(\Ker\phi)_{K}$, as claimed.
 
\medskip

\textit{Case}~2. Suppose $\dim_k\Hom_{\mathrm{Gr}/k}(G, \bbG_{a(1)}) \ge 2$. Let $\phi, \psi\colon G \to \bbG_{a(1)}$ be linearly independent morphisms. Fix an algebraically closed non-trivial field extension $K$ of $k$. Note that $\phi_{K}, \psi_{K}\colon G_K \to \mathbb G_{a(1),K}$ remain linearly independent, and hence for any pair of elements $\lambda, \mu \in K$ not both zero, $\lambda^{1/p} \phi_{K} + \mu^{1/p} \psi_{K} \ne 0$.  This implies that for any non-zero element $x$ in $H^2(\bbG_{a(1)},K)$, the element
\[ 
(\lambda^{1/p} \phi_{K} + \mu^{1/p} \psi_{K})^*(x) = \lambda\phi^*_{K}(x) + \mu\psi^*_{K}(x)
\] 
in $H^2(G,K)$ is non-zero; this follows by the semilinearity of the Bockstein map, which also explains $1/p$ in the exponents (see the
proof of \cite[Theorem 5.3]{Benson/Iyengar/Krause/Pevtsova:2015a} for more details on this formula). 

Let $M$ be a $G$-module with $\pisupp_{G}(M)=\varnothing$. The induction hypothesis implies that $M_K$ is projective when restricted to the kernel of $\lambda^{1/p} \phi_{K} + \mu^{1/p} \psi_{K}$. Thus $\lambda \phi^*_{K}(x) + \mu\psi^*_{K}(x)$ induces a periodicity isomorphism
\[ 
H^1(G,M_K) = H^1(G,M)_K \lra H^3(G,M)_K = H^3(G,M_K)\,.
\] 
As this is so for any pair $\lambda,\mu$ not both zero, the analogue of the Kronecker quiver lemma \cite[Lemma 4.1]{Benson/Carlson/Rickard:1996a} implies that
\[ 
H^1(G,M)_K = H^1(G,M_K)=0\,
\] 
Since $G$ is unipotent, this implies that $M$ is projective, as desired.
\end{proof}

\section{Finite group schemes}
\label{se:generalcase}

In this section we prove that $\pi$-support detects projectivity for any finite group scheme. It uses the following result that can be essentially found in \cite{Pevtsova:2004a}. However the pivotal identity~\eqref{eq:ind} in the proof was only justified later in \cite{Suslin:2006a}.

\begin{theorem}
\label{th:subgroup} 
Let $G \hookrightarrow G^\prime$ be an embedding of connected finite group
schemes over $k$.  If $\pi$-support detects projectivity for
$G^\prime$, then it detects projectivity for $G$.
\end{theorem}

\begin{proof} 
Let $M$ be a $G$-module such that $\pisupp_{G}(M) = \varnothing$.  By Lemma~\ref{le:vanish} and Frobenius reciprocity, it
suffices to show that $\pisupp_{G^\prime}(\ind_G^{G^\prime} M)=\varnothing$.

Consequently, we need to show that for any $\pi$-point $\alpha \colon 
K[t]/(t^p) \to KG^\prime$, the restriction $\alpha^*((\ind_G^{G^\prime}
M)_K)$ is free. By Remark~\ref{re:pi-basics}, we may assume that $\alpha$ 
factors through some quasi-elementary subgroup scheme $\mcE^\prime \le G^\prime_K$
defined over $K$. Since induction commutes with extension of scalars and we are only 
going to work with one $\pi$-point at a time, we may extend scalars and assume that $k=K$. 
 
Let $ \mcE = \mcE^\prime \cap G \le G^\prime$ (this can be the trivial group scheme).  Let $\Lambda = \End_{k}(M)$.   Since $\mcE$ is quasi-elementary, the assumption on $M$ together with Theorem~\ref{thm:Dade} imply that $M{\da_{\mcE}}$ is free. Hence,  Lemma~\ref{le:end-projectivity}(2) implies that $\Lambda{\da_{\mcE}}$ is free.   

Consider the adjunction isomorphism
\[ \Hom_{\mcE}(\ind_{G}^{G^{\prime}} \Lambda,
\Lambda) \cong \Hom_{\mcE^\prime}(\ind_{G}^{G^{\prime}} \Lambda,
\ind_{\mcE}^{\mcE^\prime}\Lambda)\,,
\]
and let 
\[ 
\theta\colon \ind_{G}^{G^{\prime}}{\Lambda}\lra \ind_{\mcE}^{\mcE^\prime}\Lambda
\] 
be the homomorphism of $\mcE^\prime$-modules which corresponds to the standard evaluation map $\epsilon_\Lambda: \ind_{G}^{G^{\prime}}\Lambda \lra \Lambda$ (see \cite[3.4]{Jantzen:2003a}) considered as a map of $\mcE$-modules.  By \cite[pp.~216--217]{Suslin:2006a}, the map $\theta$ is surjective and the ideal $I=\Ker \theta$ is nilpotent.

Indeed, it is shown in \cite{Suslin:2006a} that
\begin{equation}
\label{eq:ind} 
\ind_{\mcE}^{\mcE^\prime} \Lambda 
\cong k[\mcE^\prime/\mcE] \otimes_{k[G^\prime/G]}
\ind_{G}^{G^{\prime}} \Lambda
\end{equation} 
with the map $\theta$ induced by the extension of scalars from $k[G^\prime/G]$ to $k[\mcE^\prime/\mcE]$. Hence, the surjectivity follows from the fact that $\mcE^\prime/\mcE \to G^\prime/G$ is a closed embedding, see \cite[Thm. 5.3]{Suslin:2006a}, and the nilpotency of $I$ follows from the fact that $k[G^\prime/G]$ is a local artinian ring.

Consequently, we have an exact sequence of $\mcE^\prime$-modules 
\[ 
0 \lra I \lra \ind_{G}^{G^{\prime}} \Lambda \xra{\ \theta\ } \ind_{\mcE}^{\mcE^\prime} \Lambda \lra 0
\]
where $I$ is a nilpotent ideal and $\ind_{\mcE}^{\mcE^\prime} \Lambda$ 
is projective since $\Lambda$ is projective as an $\mcE$-module. 
The  exact sequence in cohomology now implies that any
positive degree element in $H^*(k[t]/(t^p), \ind_{G}^{G^{\prime}} \Lambda)$ 
is nilpotent, where the action of $t$ is via the $\pi$-point $\alpha\colon 
k[t]/(t^p) \to k\mcE^\prime \to kG^\prime$.

Note that the linear action of $k$ on $\ind_G^{G^\prime} M$ factors as follows: 
\[ 
\xymatrix{k \otimes_k \ind_{G}^{G^{\prime}} M \ar[r] &
\ind_{G}^{G^{\prime}} \Lambda \otimes_k
\ind_{G}^{G^{\prime}} M\ar[r] & 
\ind_{G}^{G^{\prime}} (\Lambda \otimes_k M) \ar[r] & 
\ind_{G}^{G^{\prime}} M.}
\]
So the Yoneda action of $H^*(k[t]/(t^p),k)$ on $H^*(k[t]/(t^p), \ind_{G}^{G^{\prime}}M)$ factors through the action of  $H^*(k[t]/(t^p), \ind_{G}^{G^{\prime}} \Lambda)$.  We conclude that  $H^{\ges 1}(k[t]/(t^p),k)$ acts nilpotently. On the other hand, the action of a generator in degree 2 for $p>2$ (or degree 1 for $p=2$) induces the periodicity isomorphism on $H^*(k[t]/(t^p), \ind_{G}^{G^{\prime}} M)$. 
Hence, $H^{\ges 1}(k[t]/(t^p), \ind_{G}^{G^{\prime}}{M})=0$, and therefore the equivalence class of $\alpha$ is not in $\pisupp_{G^\prime}(\ind_G^{G^\prime}{M})$. Since $\alpha$ was any $\pi$-point, the statement follows.
\end{proof}

We also require the following detection criterion; see \cite[Theorem1.6]{Pevtsova:2002a}.

\begin{theorem}
\label{th:lpt} Let $\mcG$ be a connected reductive algebraic group
over $k$ and let $\mcG_{(r)}$ be its $r$-th Frobenius kernel.  If $M$
is a $\mcG_{(r)}$-module such that for any field extension $K/k$ and
any embedding of group schemes $\mathbb G_{a(r),K}\hookrightarrow
\mcG_{(r),K}$, the restriction of $M_K$ to $\mathbb G_{a,K}$ is
projective, then $M$ is projective as a $\mcG_{(r)}$-module.\qed
\end{theorem}

We come now to the central result of the first part of this article.

\begin{theorem}
\label{th:main} 
Let $G$ be a finite group scheme over $k$. A $G$-module $M$ is projective if and only if $\pisupp_{G}(M)=\varnothing$.
\end{theorem}

\begin{proof} 
Assume $M$ is projective and let  $\alpha\colon K[t]/(t^{p})\to KG$ be a $\pi$-point of $G$. The $G_{K}$-module $M_{K}$ is then projective, and hence so is the $K[t]/(t^{p})$-module $\alpha^{*}(M_{K})$, for $\alpha$ is flat when viewed as a map of algebras. Thus $\pisupp_{G}(M)=\varnothing$.

The proof of the converse builds up in a number of steps.

\subsection*{Frobenius kernels}
Suppose $G:=\Gr$, the $r$th Frobenius kernel of a connected reductive group $\mcG$ over $k$. Let $M$ be a $G$-module with $\pisupp_{G}(M)=\varnothing$. For any field extension $K/k$ and embedding $\phi\colon \mathbb G_{a(r),K} \hookrightarrow \mcG_{(r),K}$ of group schemes over $K$, one then has $\pisupp_{\mcG_{(r),K}}(M_{K})=\varnothing$, and hence it follows that
\[ 
\pisupp_{\bbG_{a(r),K}}(\phi^*(M_K))=\varnothing\,.
\] 
Theorem~\ref{thm:Dade} then implies that $\phi^*(M_K)$ is projective as a $\mathbb G_{a(r),K}$-module. It remains to apply
Theorem~\ref{th:lpt}.

\subsection*{Connected finite group schemes} 
This case is immediate from the preceding one and Theorem~\ref{th:subgroup} since any connected finite group scheme can be embedded into $\GL_{n(r)}$ for some positive integers $n, r$; see \cite[3.4]{Waterhouse:1979a}.

\subsection*{$G\cong G^\circ \times (\bbZ/p)^{r}$ where $G^\circ$ is the connected component at the identity}

Let $M$ be a $G$-module with $\pisupp_{G}(M) =\varnothing$. Let $\mcE = (\bbG_{a(1)})^{\times r}$  and observe that the $k$-algebras $kG$ and $k (G^\circ \times \mcE)$ are isomorphic, and hence so are $H^*(G, M)$ and $H^*(G^\circ \times \mcE, M)$. Moreover, $(\bbZ/p)^{r}$ and $\mcE$ are both unipotent abelian group schemes, so the maximal unipotent abelian subgroup schemes of $G$ and $G^{\circ}\times\mcE$, and hence also their $\pi$-points, are in bijection. In summary: $\pisupp_{G^\circ \times \mcE}(M) = \pisupp_G(M)=\varnothing$. Since we have verified already that $\pi$-support detects projectivity for connected finite group schemes, and $G^{\circ}\times \mcE$ is one such, one gets the equality below
\[
H^{i}(G,M)\cong H^{i}(G^\circ \times \mcE, M) = 0 \quad\text{for $i\ge 1$.}
\]
It remains to recall Lemma~\ref{le:vanish} to deduce that $M$ is projective as a $G$-module.

\subsection*{General finite group schemes} Extending scalars, if needed, we may assume that $k$ is algebraically closed.  The proof is by induction on $\dim_k k[G]$. The base case is trivial.  Suppose the theorem holds for all proper subgroup schemes of $G$. Let $G = G^\circ \rtimes \pi_{0}(G)$ where $G^\circ$ is the connected component at the identity and $\pi_{0}(G)$ is the (finite) group of connected components.  If the product is direct and $\pi_{0}(G)$ is elementary abelian, then we have already verified that the desired result holds for $G$. We may thus assume that this is not the case; this implies that for any elementary abelian subgroup $E <\pi_{0}(G)$, the subgroup scheme $(G^\circ)^E \times E$ is a proper subgroup scheme of $G$.

If follows from the Quillen stratification for the space of
equivalence classes of
$\pi$-points~\cite[4.12]{Friedlander/Pevtsova:2007a} that any
$\pi$-point for $G$ is equivalent to one of the form
$\alpha\colon K[t]/t^{p}\to KG$ that factors through
$((G^\circ)^E \times E)_K < (G^\circ \rtimes \pi_{0}(G))_K$, where
$E < \pi_{0}(G)$ is an elementary abelian subgroup. Thus, the
hypotheses of Theorem~\ref{th:principle} holds, and we can conclude
that $M$ is projective, as needed.
\end{proof}

\begin{remark}
\label{re:fpnot} 
The implication that when the $\pi$-support of a $G$-module $M$ is empty it is projective is the content of \cite[Theorem~5.3]{Friedlander/Pevtsova:2007a}.  However, the proof given in \cite{Friedlander/Pevtsova:2007a} is incorrect. The problem occurs in the third paragraph of the proof where what is asserted
translates to: the $\pi$-support of $\End_{k}(M)$ is contained in the $\pi$-support of $M$.  This is not so; see \cite[Example~6.4]{Benson/Iyengar/Krause/Pevtsova:2015a}. What is true is that the $\pi$-\textit{cosupport} of $\End_{k}(M)$ is contained in the $\pi$-support of $M$, by Theorem~\ref{th:tensor-and-hom-pi}. This is why it is useful to consider cosupports even if one is interested only in supports.
\end{remark}

Chouinard~\cite[Corollary~1.1]{Chouinard:1976a} proved that a module $M$ over a finite group $G$ is projective if its restriction to every elementary abelian subgroup of $G$ is projective. This result is fundamental to the development of the theory of support varieties for finite groups.  For finite group schemes Theorem~\ref{th:main} yields the following analogue of Chouinard's theorem. There are two critical differences: one has to allow for field extensions and there are infinitely many subgroup schemes involved.

\begin{corollary}
\label{co:Chouinard}
Let $G$ be a finite group scheme over $k$. A $G$-module $M$ is projective if for every field extension $K/k$ and quasi-elementary subgroup scheme $\mcE$ of $G_{K}$, the $\mcE$-module $(M_{K})\da_{\mcE}$ is projective. 
\end{corollary}

\begin{proof}
As noted in Remark~\ref{re:pi-basics}, every $\pi$-point factors through some $\mcE$ as above, so if $(M_{K})\da_{\mcE}$ is projective for each such $\mcE$, it follows that $\pisupp_{G}(M)=\varnothing$, and hence that $M$ is projective, by Theorem~\ref{th:main}.
\end{proof}

\begin{remark}
\label{re:cosupp} 
All the steps in the proof of Theorem~\ref{th:main} except for the one dealing with Frobenius kernels, Theorem~\ref{th:lpt}, work equally well for $\pi$-cosupport: namely, they can be used with little change to show that if $\picosupp M = \varnothing$ then $M$ is projective. Explicitly, the following changes need to be made:
\begin{enumerate}
\item Theorem~\ref{th:principle}: Simply replace $\pi$-support with $\pi$-cosupport.
\item Theorem~\ref{th:unip}: In the proof of Case 2, use \cite[Lemma 5.1]{Benson/Iyengar/Krause/Pevtsova:2015a} which is an analogue for cosupports of the Kronecker quiver lemma.
\item Theorem~\ref{th:subgroup}: The proof 
  carries over almost verbatim. One replaces the extension $M_K$ with
  coextension $M^K$ and uses repeatedly that coextension commutes with
  induction for finite group
  schemes~\cite[Lemma~2.2]{Benson/Iyengar/Krause/Pevtsova:2015a}.
\end{enumerate}

The trouble with establishing the analogue of Theorem~\ref{th:lpt} for cosupports can be pinpointed to the fact that the induction functor $\ind\colon \Mod H \to \Mod G$ does not commute with coextension of scalars for general affine group schemes. Even worse, when $G$ is not finite and $K/k$ is of infinite degree, given a $G$-module $M$ there is no natural action of $G_K$ on $M^K$.

In Part~\ref{part:applications} we prove that $\pi$-cosupport detects
projectivity, taking an entirely different approach. This uses the
support detection theorem in an essential way.
\end{remark}

\part{Minimal localising subcategories}
\label{part:minimality} 
Let $G$ be a finite group scheme over a field $k$. From now on we consider the stable module category $\StMod G$ whose construction and basic properties were recalled in Section~\ref{se:bik}. For each $\fp$ in $\Proj H^{*}(G,k)$, we focus on the subcategory
$\gam_{\fp}(\StMod G)$ consisting of modules with support in $\{\fp \}$.  These are precisely the modules whose cohomology is
$\fp$-local and $\fp$-torsion.

This part of the paper is dedicated to proving that $\gam_{\fp}(\StMod G)$ is \emph{minimal}, meaning that it contains no proper non-zero tensor ideal localising subcategories. As noted in Section~\ref{se:bik}, this is the crux of the classification of the tensor ideal localising
subcategories of $\StMod G $.

For closed points in $\Proj H^{*}(G,k)$ the desired minimality is verified in Section~\ref{se:cohomological-support}. The general case
is settled in Section~\ref{se:passage-to-closed-points}, by reduction to a closed point. The key idea here is to construct good generic
points for projective varieties. The necessary commutative algebra is developed in Section~\ref{se:generic-points}.

\section{Support equals $\pi$-support}
\label{se:cohomological-support} Henceforth it becomes necessary to
have at our disposal the methods developed in
\cite{Benson/Iyengar/Krause:2008a, Benson/Iyengar/Krause:2011a} and
recalled in Section~\ref{se:bik}. We begin by establishing that the
$\pi$-support of a $G$-module coincides with its support. Using this,
we track the behaviour of supports under extensions of scalars and
verify that for a closed point $\fp$ the tensor ideal localising
subcategory $\gam_{\fp}(\StMod G)$ is minimal.
 
\begin{theorem}
  \label{th:pisupp=bik} 
  Let $G$ be a finite group scheme defined over
  $k$. Viewed as subsets of $\Proj H^{*}(G,k)$ one has
  $\pisupp_{G}(M) = \supp_{G}(M)$ for any $G$-module $M$.
\end{theorem}

\begin{proof} 
  From \cite[Proposition~6.6]{Friedlander/Pevtsova:2007a} one gets
  that $\pisupp_{G}(\gam_{\fp} k) =\{\fp\}$. Given this, the tensor
  product formula (Theorem~\ref{th:tensor-and-hom-pi}) and the
  detection of projectivity (Theorem~\ref{th:main}), the calculation
  is purely formal; see the proof of \cite[Theorem  6.1]{Benson/Iyengar/Krause/Pevtsova:2015a}.
\end{proof}
 
The preceding result reconciles two rather different points of view of
support and so makes available a panoply of new tools for studying
representation of finite group schemes. The next result, required in
Section~\ref{se:passage-to-closed-points}, well illustrates this
point.
 
\begin{proposition}
\label{pr:basechange-supp} 
Let $G$ be a finite group scheme over $k$, let $K/k$ be an extension of fields, and $\rho\colon \Proj
H^{*}(G_{K},K)\to \Proj H^{*}(G,k)$ the induced map.
\begin{enumerate}[\quad\rm(1)]
\item $\supp_{G_{K}}(M_{K}) = \rho^{-1}(\supp_{G}(M))$ for any
$G$-module $M$.
\item $\supp_{G}(N\da_{G})= \rho(\supp_{G_{K}}(N))$ for any
$G_{K}$-module $N$.
\end{enumerate}
\end{proposition}

\begin{proof} The equality in (1) is clear for $\pi$-supports; now
recall Theorem~\ref{th:pisupp=bik}.

We deduce the equality in (2) by applying twice
Proposition~\ref{pr:change-cat-ring}. The action of $H^{*}(G_{K},K)$
on $\StMod G_{K}$ induces an action also of $H^{*}(G,k)$ via
restriction of scalars along the homomorphism $K\otimes_{k}-\colon
H^{*}(G,k)\to H^{*}(G_{K},K)$. Applying
Proposition~\ref{pr:change-cat-ring} to the functor
$(\id,K\otimes_{k}-)$ on $\StMod G_{K}$ yields an equality
\[ \supp_{H^{*}(G,k)}(N) = \rho(\supp_{G_{K}}(N))\,.
\] Next observe that the restriction functor $(-)\da_{G}$ is
compatible with the actions of $H^{*}(G,k)$.  Also, $(-)\da_{G}$ is
exact, preserves set-indexed coproducts and products, and is faithful
on objects. Everything is obvious, except the last property. So
suppose $N$ is a $G_{K}$-module such that $N\da_{G}$ is
projective. Then, for any simple $G$-module $S$ and integer $i\ge 1$
one has
\[ \Ext^{i}_{G_{K}}(K\otimes_{k}S,N) \cong \Ext^{i}_{G}(S,N\da_{G}) =
0
\] Since any simple $G_{K}$-module is a direct summand of
$K\otimes_{k}S$, for some choice of $S$, it follows that $N$ is
projective, as desired.

Now we apply
Proposition~\ref{pr:change-cat-ring}  to the functor
\[ ((-)\da_{G},\id_{H^{*}(G,k)})\colon \StMod G_{K} \lra \StMod G
\] and obtain the equality 
\[ \supp_{G}(N\da_{G})=\supp_{H^{*}(G,k)}(N) \,.\]
In conjunction with the one
above, this gives (2).
\end{proof}

We can now begin to address the main task of this part of the paper.

\begin{proposition}
\label{pr:minimality-closed} When $\fm$ is a closed point of $\Proj
H^{*}(G,k)$, the tensor ideal localising subcategory
$\gam_{\fm}(\StMod G)$ of $\StMod G$ is minimal.
\end{proposition}

\begin{proof} It suffices to verify that the $G$-module $\Hom_k(M,N)$
is not projective for any non-zero modules $M,N$ in $\gam_{\fm}(\StMod
G)$; see  \cite[Lemma~3.9]{\bik:2011b}.

A crucial observation is that since $\fm$ is a closed point, it is in $\pisupp_G(M)$ if and only if it is in $\picosupp_G(M)$ for any
$G$-module $M$; see Lemma~\ref{le:maxsupp=maxcosupp}. This will be used (twice) without comment in what follows.  For any non-zero
modules $M,N$ in $\gam_{\fm}(\StMod G)$ Theorem~\ref{th:tensor-and-hom-pi} yields
\[ 
\picosupp_{G}(\Hom_k(M,N)) = \pisupp_G(M) \cap \picosupp_G(N) = \{\fm\}\,.
\] 
Thus, $\fm$ is also in the support of $\Hom_{k}(M,N)$. It remains to recall Theorem~\ref{th:main}.
\end{proof}

\section{Generic points in graded-commutative algebras}
\label{se:generic-points} 
A standard technique in classical algebraic geometry is to ensure that
irreducible varieties have generic points by enlarging their field of
definition. For affine varieties this amounts to the following: Given
a prime ideal $\fp$ in an algebra $A$ finitely generated  over
a field $k$, there is an extension of fields $K/k$ such that in the
ring $B:=A\otimes_{k}K$ there is a maximal ideal $\fm$ lying over
$\fp$, that is to say, $\fm\cap A=\fp$.  In
Section~\ref{se:passage-to-closed-points} we need a more precise
version of this result, namely that there is such a $K$ where $\fm$ is
cut out from $B/\fp B$, the fiber over $\fp$, by a complete
intersection in $B$; also, we have to deal  with projective
varieties. This is 
what is achieved in this section; see
Theorem~\ref{th:generic-point}. The statement and its proof require
some care, for in our context the desired property holds only outside
a hypersurface.

Let $B$ be a graded-commutative ring: a graded abelian group $B=\{B^{i}\}_{i\in\bbZ}$ with an associative product satisfying
$a\cdot b = (-1)^{|a||b|}b\cdot a$ for all elements $a,b$ in $B$, where $|\ |$ denotes degree. We consider only homogenous elements of graded objects.

\begin{definition}
\label{de:weakly-regular} Let $N$ be a graded $B$-module. Mimicking
\cite[Definition 1.1.1]{Bruns/Herzog:1998a}, we say that a sequence
$\bsb:=b_{1},\dots,b_{n}$ of elements in $B$ is a \emph{weak
$N$-sequence} if $b_{i}$ is not a zerodivisor on
$N/(b_{1},\dots,b_{i-1})N$, for $i=1,\dots,n$. We drop the adjective
``weak'' if, in addition, $\bsb N\ne N$ holds.
\end{definition}

\begin{lemma}
\label{le:regular-sequence} Let $A\to B$ be a homomorphism of
graded-commutative rings and $\bsb:=b_{1},\dots,b_{n}$ a weak
$B$-sequence. If the $A$-module $B/(b_{1},\dots,b_{i})B$ is flat for
each $i=1,\dots,n$, then $\bsb$ is a weak $(M\otimes_{A}B)$-sequence
for each graded $A$-module $M$.
\end{lemma}

\begin{proof} Set $\bsb_{\les i}:=b_{1},\dots,b_{i}$ for
$i=0,\dots,n$. For $i=1,\dots,n$, since $b_{i}$ is not a zero divisor on
$B/(\bsb_{\les i-1})$, one gets the following exact sequence of of
graded $B$-modules.
\[ 0\lra \frac B{(\bsb_{\les i-1})B} \xra{\ b_{i}\ } \frac
B{(\bsb_{\les i-1})B} \lra \frac B{(\bsb_{\les i})B}\lra 0
\] Since $B/(\bsb_{\les i})$ is flat as an $A$-module, applying
$M\otimes_{A}-$ to the exact sequence above and noting that
$M\otimes_{A} B/(\bsb_{\les i-1})$ is naturally isomorphic to
$(M\otimes_{A}B)/(\bsb_{\les i-1})$, one gets the following exact sequence.
\[ 0\lra \frac {M\otimes_{A}B}{(\bsb_{\les i-1})(M\otimes_{A}B)}
\xra{\ b_{i}\ } \frac {M\otimes_{A}B}{(\bsb_{\les
i-1})(M\otimes_{A}B)} \lra \frac {M\otimes_{A}B}{(\bsb_{\les
i})(M\otimes_{A}B)}\lra 0
\] This is the desired conclusion.
\end{proof}

\subsection*{A model for localisation} To prepare for the next step,
we recall some basic properties of the kernel of a diagonal map. Let
$k$ be a field and $k[\bsx]$ a polynomial ring over $k$ in
indeterminates $\bsx:=x_{0},\dots,x_{n}$ of the same degree. Let
$\bst:=t_{1},\dots,t_{n}$ be indeterminates over $k$ and $k(\bst)$ the
corresponding field of rational functions, and consider the
homomorphism of $k$-algebras
\[ 
\mu \colon k(\bst)[\bsx]\twoheadrightarrow
k\left(\frac{x_{1}}{x_{0}},\dots,\frac{x_{n}}{x_{0}}\right)[x_{0}]
\quad \text{where $\mu(t_{i})=\frac{x_{i}}{x_{0}}$ for each $i$.}
\] 
The range of $\mu$ is viewed as a subring of the field of rational functions in $\bsx$.

\begin{lemma}
\label{le:diagonal} The ideal $\Ker(\mu)$ is generated by
$x_{1}-x_{0}t_{1},\dots,x_{n}-x_{0}t_{n}$, and the latter is a
$k(\bst)[\bsx]$-sequence
\end{lemma}

\begin{proof} It is clear that the kernel of $\mu$ is generated by the
given elements. That these elements form a $k(\bst)[\bsx]$-sequence
can be readily verified by, for example, an induction on $n$. Another
way is to note that they are $n$ elements in a polynomial ring and the
Krull dimension of $k(\bst)[\bsx]/\Ker(\mu)$, is one; see
\cite[Theorem~2.1.2(c)]{Bruns/Herzog:1998a}.
\end{proof}

Let now $A$ be a graded-commutative $k$-algebra, and
$\bsa:=a_{0},\dots,a_{n}$ an \emph{algebraically independent} set over
$k$, with each $a_{i}$ of the \emph{same degree}. Observe that the
following subset of $A$ is multiplicatively closed.
\begin{equation}
\label{eq:localising-set} U_{\bsa}:= \{f(a_{0},\dots,a_{n})\mid
\text{$f$ a non-zero homogeneous polynomial}\}
\end{equation} 
The algebraic independence of $\bsa$ is equivalent to the condition that $0$ is not in $U_{\bsa}$.  For example, $U_{a}$ is the multiplicatively closed subset $\cup_{i\ges 0}ka^{i}$. For any $A$-module $M$ one has the localisation at $U_{\bsa}$, namely equivalence classes of fractions
\[ 
U_{\bsa}^{-1}M:= \big\{ \big[\frac mf\big] \mid m\in M \text{ and }
f\in U_{\bsa} \big\}
\] 
The result below provides a concrete realisation of this localisation.

\begin{lemma}
\label{le:localisation} Let $\bst:=t_{1},\dots,t_{n}$ be
indeterminates over $k$ and $k(\bst)$ the corresponding field of
rational functions. Set $B:=A\otimes_{k}k(\bst)$ and
$b_{i}:=a_{i}-a_{0}t_{i}$, for $i=1,\dots,n$. The following statements
hold.
\begin{enumerate}[\quad\rm(1)]
\item The canonical map $A\to B/(\bsb)$ of $k$-algebras induces an
isomorphism
\[ U_{\bsa}^{-1}A\xra{\ \cong\ } U_{a_{0}}^{-1}(B/(\bsb))\,.
\]
\item 
$\bsb$ is a weak $U_{a_{0}}^{-1}(M\otimes_{k}k(\bst))$-sequence for any graded $A$-module $M$.
\end{enumerate}
\end{lemma}

\begin{proof} We first verify the statements when $A=k[\bsx]$, a
polynomial ring over $k$ in indeterminates $\bsx:=x_{0},\dots,x_{n}$
of the same degree, and $a_{i}=x_{i}$ for each $i$. Then
$B=k(\bst)[\bsx]$, the polynomial ring over the same indeterminates,
but over the field $k(\bst)$, and $U_{x_{0}}^{-1} k(\bst)[\bsx]$ can
be naturally identified with $k(\bst)[\bsx,x_{0}^{-1}]$.

Consider the commutative diagram of morphisms of graded $k$-algebras
\[
\begin{gathered} \xymatrixcolsep{1pc} \xymatrix{ k[\bsx] \ar@{->}[d]
\ar@{->}[r] & k(\bst)[\bsx] \ar@{->}[r] & k(\bst)[\bsx,x_{0}^{-1}]
\ar@{->>}[d]^{U_{x_{0}}^{-1}\mu} \\ U_{\bsx}^{-1}k[\bsx]
\ar@{->}[rr]^-{\cong} && k(x_{1},\dots,x_{n})[x^{\pm 1}_{0}] }
\end{gathered} \quad \text{where $\mu(t_{i})=\frac{x_{i}}{x_{0}}$.}
\] The unlabeled arrows are all canonical inclusions and the
isomorphism is obvious. It follows from Lemma~\ref{le:diagonal} that
$\Ker(U_{x_{0}}^{-1}\mu)$ is the ideal generated by
$\{x_{i}-x_{0}t_{i}\}_{i=1}^{n}$. This justifies the assertion in (1).

As to (2), since $x_{1}-x_{0}t_{1},\dots,x_{n}-x_{0}t_{n}$ is a
$k(\bst)[\bsx]$-sequence, by Lemma~\ref{le:diagonal}, it is also a
weak $k(\bst)[\bsx,x_{0}^{-1}]$-sequence. Moreover, arguing as above
one gets that there is an isomorphism of graded rings
\[
\frac{k(\bst)[\bsx,x_{0}^{-1}]}{(x_{1}-x_{0}t_{1},\dots,x_{i}-x_{0}t_{i})}
\cong k(x_{1},\dots, x_{i},t_{i+1},\dots,t_{n})[\bsx,x^{-1}_{0}]
\] for each $1\le i\le n$. In particular, these are all flat as
modules over $k[\bsx,x^{-1}_{0}]$, for they are obtained by
localisation followed by an extension of scalars. Thus
Lemma~\ref{le:regular-sequence} applied to the morphism
$k[\bsx,x^{-1}_{0}]\to k(\bst)[\bsx,x^{-1}_{0}]$, yields (2).

This completes the proof of the result when $A=k[\bsx]$.

\medskip

The desired statements for a general $A$ follow readily by base
change. Indeed, consider the morphism of graded $k$-algebras
$k[\bsx]\to A$ given by the assignment $x_{i}\mapsto a_{i}$, for each
$i$. It is easy to see then that $B\cong
k(\bst)[\bsx]\otimes_{k[\bsx]}A$, so that applying
$-\otimes_{k[\bsx]}A$ to the isomorphism
\[ U_{\bsx}^{-1}k[\bsx] \xra{\ \cong\ } k(x_{1},\dots,x_{n})[x^{\pm
1}_{0}]
\] gives the isomorphism in (1). As to (2), viewing a graded
$A$-module $M$ as an module over $k[\bsx]$ via restriction of scalars,
and applying the already established result for $k[\bsx]$ gives the
desired conclusion.
\end{proof}

Let $k$ be a field and $A=\{A^{i}\}_{\ges0}$ a finitely generated
graded-commutative $k$-algebra with $A^{0}=k$. As usual $\Proj A$
denotes the collection of homogeneous prime ideals in $A$ not
containing $A^{\ges 1}$. Given a point $\fp$ in $\Proj A$, we write
$k(\fp)$ for the localisation of $A/\fp$ at the set of non-zero
homogenous elements of $A/\fp$. Note that $k(\fp)$ is a graded field
and its component in degree zero is the field of functions at $\fp$.

\begin{definition}
  \label{de:noether-normalisation} Let $A$ be a domain and set
  $Q:=k((0))$, the graded field of fractions of $A$. We say that
  elements $\bsa:=a_{0},\dots,a_{n}$ in $A$ give a \emph{Noether
    normalisation} of $A$ if the $a_{i}$ all have the same positive
  degree, are algebraically independent over $k$, and $A$ is a
  finitely generated module over the subalgebra $k[\bsa]$. Noether
  normalisations exist; see, for example,
  \cite[Theorem~1.5.17]{Bruns/Herzog:1998a}, noting that, in the
  language of \emph{op.~cit.}, a sequence $a_{0},\dots,a_{n}$ is a
  system of parameters for $A$ if and only so is the sequence
  $a_{0}^{e_{0}},\dots,a_{n}^{e_{n}}$, for any positive integers
  $e_{0},\dots,e_{n}$.

Observe that if $\bsa$ is a Noether normalisation of $A$, then the set
$\{a_{1}/a_{0},\dots, a_{n}/a_{0}\}$ is a transcendence basis for the
extension of fields $k\subseteq Q_{0}$.
\end{definition}

The result below, though not needed in the sequel, serves to explain
why in constructing generic points it suffices to enlarge the field of
definition to function fields of Noether normalisations. 

\begin{lemma}
\label{le:field-of-fractions} The inclusion $A\to Q$ induces an
isomorphism $U_{\bsa}^{-1}A\xra{\cong} Q$.
\end{lemma}

\begin{proof} 
By the universal property of localisations, it suffices to verify that $U_{\bsa}^{-1}A$ is a graded field. Recall that
$U_{\bsa}$ is the set of homogenous elements in $k[\bsa]\setminus \{0\}$. By definition, $A$ is
finitely generated as a module over $k[\bsa]$, so $U_{\bsa}^{-1}A$ is
finitely generated as a module over $U_{\bsa}^{-1}k[\bsa]$. The latter
is a graded field, hence so is the former, as it is a domain.
\end{proof}

Fix a point $\fp$ in $\Proj A$ and elements $\bsa:=a_{0},\dots,a_{n}$ in $A$ whose residue classes modulo $\fp$ give a Noether normalisation of $A/\fp$; see Definition~\ref{de:noether-normalisation}. Let $K:=k(\bst)$, the field of rational functions in indeterminates
$\bst:=t_{1},\dots,t_{n}$ over $k$. Set
\[
 B:= A\otimes_{k}K \quad\text{and}\quad b_{i}:= a_{i} - a_{0}t_{i} \quad\text{for $i=1,\dots,n$}.
\] 
Thus $B$ is a $K$-algebra. The next result is probably well-known but we were unable to find an adequate reference. Recall that a point $\fm$ in $\Proj B$ is \emph{closed} if it is maximal with respect to inclusion; equivalently, the Krull dimension of $B/\fm$ is one.

\begin{theorem}
\label{th:generic-point} 
Set $\fm:=\sqrt{(\fp,b_{1},\dots,b_{n})B}$. The following statements hold.
\begin{enumerate}[\quad\rm(1)]
\item $\fm$ is prime ideal in $B$ and defines a closed point in $\Proj B$.
\item $\fm\cap A = \fp$.
\item $b_{1},\dots,b_{n}$ is a weak $U^{-1}_{a_{0}}B$-sequence.
\end{enumerate}
\end{theorem}

\begin{proof} 
Note that the set $\bsa$ is algebraically independent over $k$, since it has  that property modulo $\fp$. Thus (3) is a special case of
Lemma~\ref{le:localisation}(2).

As to (1) and (2), replacing $A$ by $A/\fp$, we can assume $A$ is a domain with Noether normalisation $\bsa:=a_{0},\dots,a_{n}$ and
$\fp=(0)$. Set $\bsb:=b_{1},\dots,b_{n}$, so that $\fm = \sqrt{\bsb B}$. In what follows, it will be helpful to keep in mind the following commutative diagram of homomorphisms of graded rings:
\begin{equation}
\label{eq:generic} 
\xymatrixcolsep{3pc} 
\xymatrix{ 
A \ar@{->}[d] \ar@{->}[r]^{\alpha} & B/\bsb B \ar@{->}[d]^{\beta} \ar@{->>}[r]& B/\fm \ar@{->}[d] \\
U_{\bsa}^{-1}A\ar@{->}[r]^-{\cong} & U_{a_{0}}^{-1}(B/\bsb B)\ar@{->>}[r]^{\gamma} & U_{a_{0}}^{-1}(B/\fm) 
}
\end{equation} 
The map $\alpha$ is the composition $A\to B\to B/\bsb B$ while $\beta$ is localisation at $U_{a_{0}}$. The isomorphism is by Lemma~\ref{le:localisation}. Since $A$ is a domain, the vertical map on the left is one-to-one, and hence so is the map $\alpha$. This proves that $\bsb B\cap A=(0)$, but we need more. In what follows, $\Spec B$ is the collection of homogeneous prime ideals of $B$.

\medskip

(1) Since $\fm=\sqrt{\bsb B}$, the desired result follows from statements below:
\begin{enumerate}[\quad\rm(i)]
\item $\height(\fq)=n$ for any $\fq\in \Spec B$ minimal over $\bsb B$.
\item $a_{0}\not\in \fq$ for any $\fq\in \Spec B$ minimal over $\bsb B$.
\item $\bsb B$ has exactly one prime ideal minimal over it.
\end{enumerate}

\medskip

(i) Since $\bsb B$ is generated by $n$ elements $\height(\fq)\le n$ for each $\fq$ minimal over $\bsb B$,  by the Krull Height Theorem~\cite[Theorem~A.1.]{Bruns/Herzog:1998a}.  On the other hand, by construction, $B$ is finitely generated as a module over its subalgebra $K[\bsa]$. Notice that $\bsb$ is contained in $K[\bsa]$, so it follows that $B/\bsb  B$ is a finitely generated module over $K[\bsa]/(\bsb)$.  Since $\bsa$ is algebraically independent, $K[\bsa]/(\bsb)$ is isomorphic to $k(a_{1}/a_{0},\dots,a_{n}/a_{0})[a_{0}]$, see Lemma~\ref{le:diagonal}, and hence of Krull dimension one. It follows that $\dim(B/\bsb B)\le 1$, and therefore that $\height(\fq)\ge n$, because $B$ is a catenary ring. This completes the proof of (i)

\medskip

(ii) Suppose $a_{0}$ is in some $\fq\in\Spec B$ minimal over $\bsb B$. Then $\fq$ contains the ideal $(a_{0},\dots,a_{n})$, because
$\bsb \subseteq\fq$. Recall that $B$ is finitely generated as a module over its subalgebra $K[\bsa]$. Thus, $B/\bsa B$ is finitely generated
as a module over $ K[\bsa]/(\bsa)\cong K$ and hence the Krull dimension of $B/\bsa B$ is zero. Said otherwise, the radical of $\bsa$ equals
$B^{\ges 1}$, the (unique) homogeneous maximal ideal of $B$. This justifies the first equality below.
\[ 
\height(\fq)\geq \height(\bsa)=\dim B = n+1
\] 
The inequality holds because $\fq\supseteq\fa$. As to the second equality: $B$ is a domain that is a finitely generated module over
$K[\bsa]$, which is of Krull dimension $n+1$. The resulting inequality $\height(\fq)\ge n+1$ contradicts the conclusion of (i). This settles (ii).

\medskip

(iii) The elements in $\Spec {(U_{a_{0}}^{-1}B)}$ minimal over $\bsb U_{a_{0}}^{-1}B$ are in bijection with the elements of $\Spec B$ minimal over $\bsb B$ and not containing $a_{0}$. Since $U_{a_{0}}^{-1}(B/\bsb B)$ is a domain, by \eqref{eq:generic}, and $a_{0}$ is not in any minimal prime of $\bsb B$, it follows that $\bsb B$ has only one prime ideal minimal over it, as asserted in (iii).

At this point, we have completed the proof of (1).

\medskip

(2) We have to verify that $\fm\cap A=(0)$. Since $U_{a_{0}}^{-1}(B/\bsb B)$ is a domain, by \eqref{eq:generic}, the ideal $\bsb U_{a_{0}}^{-1}B$ is prime.  By (1), the ideal $\fm$ is the unique prime minimal over $\bsb B$, so it follows that
\[
\bsb U_{a_{0}}^{-1} B = \fm U_{a_{0}}^{-1} B.
\]
Therefore, the map $\gamma$ in \eqref{eq:generic} is also an isomorphism. Consequently, the composed map $A\to B/\bsb B\to B/\fm$ is one-to-one, which is the desired result.
\end{proof}

In an earlier version of this work, we had claimed that the ideal $(\fp,b_{0},\dots,b_{n})$ in Theorem~\ref{th:generic-point} is itself prime. This need not be the case; the flaw in our argument was pointed out to us by Amnon Neeman.

\section{Passage to closed points}
\label{se:passage-to-closed-points} 
As usual let $G$ be a finite group scheme over a field $k$ of positive characteristic.  In this section we prove that for any point $\fp$ in
$\Proj H^{*}(G,k)$ the category $\gam_{\fp}(\StMod G)$ consisting of
the $\fp$-local and $\fp$-torsion $G$-modules is minimal.  The main
step in this proof is a concrete model for localisation at
multiplicatively closed subsets of the form $U_{\bsa}$;
see~\eqref{eq:localising-set}. With an eye towards future
applications, we establish a more general statement than needed for
the present purpose.

We begin by recalling the construction of Koszul objects from \cite[Definition~5.10]{Benson/Iyengar/Krause:2008a}.

\subsection*{Koszul objects} 
Each element $a$ in $H^{d}(G,k)$ defines
a morphism $k\to \Omega^{-d}k$ in $\StMod G$; we write $\kos ka$ for
its mapping cone. This is nothing but a shift of the Carlson module,
$L_{a}$, that came up in Lemma~\ref{le:nilp}.  We have opted to stick
to $\kos ka$ for this is what is used in
\cite{Benson/Iyengar/Krause:2008a, Benson/Iyengar/Krause:2011a} which
are the main references for this section.

It follows from the construction that, in $\StMod G$, there is an exact triangle
\begin{equation*}
\Omega^{d}k\xra{\ a\ } k \xra{\ q_{a}\ }
\Omega^{d}(\kos ka)\lra
\end{equation*} 
Given a sequence of elements $\bsa:=a_{1},\dots,a_{n}$ in $H^{*}(G,k)$, consider the $G$-module
\[ 
\kos k{\bsa}:= (\kos k{a_{1}})\otimes_{k}\cdots\otimes_{k} (\kos k{a_{n}})\,.
\] 
It comes equipped with a morphism in $\StMod G$
\begin{equation}
\label{eq:koszul-map} 
q_{\bsa}:=q_{a_{1}} \otimes_{k}\cdots\otimes_{k} q_{a_{n}}\colon k \lra \Omega^{d}(\kos k{\bsa})
\end{equation} 
where $d= |a_{1}| + \cdots + |a_{n}|$.  For any $G$-module $M$, set
\[ 
\kos M{\bsa}:= M\otimes_{k}(\kos k{\bsa})\,.
\] 
In the sequel, we need the following computation:
\begin{equation}
\label{eq:kos-support} 
\supp_{G}(\kos M{\bsa}) = \supp_{G}(M) \cap \mcV(\bsa)\,.
\end{equation} 
This is a special case of \cite[Lemma~2.6]{Benson/Iyengar/Krause:2011a}.

\begin{remark}
\label{re:localisation-functor} 
We say that an element $a$ in $H^{d}(G,k)$ is \emph{invertible} on a $G$-module $M$ if the canonical map $M\xra{a}\Omega^{-d}M$ in $\StMod G$ is an isomorphism. This is equivalent to the condition that $\kos Ma=0$. A subset $U$ of $H^{*}(G,k)$ is said to be invertible on $M$ if each element in it has that property.

Fix a multiplicatively closed subset $U$ of $H^{*}(G,k)$ and set
\[ 
\mcZ(U):=\{\fp \in \Spec H^{*}(G,k)\mid \fp \cap U \not=\varnothing\}.
\] 
This subset is specialisation closed. The associated localisation functor $L_{\mcZ(U)}$, whose construction was recalled in Section~\ref{se:bik}, is characterised by the property that for any $G$-modules $M$ and $N$, with $M$ finite dimensional, the
induced morphism
\[ 
\sHom^{*}_{G}(M,N)\lra \sHom^{*}_{G}(M,L_{\mcZ(U)}N)
\] 
of graded $H^*(G,k)$-modules is localisation at $U$; see, for example, \cite[Theorem~3.3.7]{Hovey/Palmieri/Strickland:1997a}.  In
particular, the set $U$ is invertible on $L_{\mcZ(U)}N$. For this reason, in what follows we  use the more suggestive notation
$U^{-1}N$ instead of $L_{\mcZ(U)}N$.
\end{remark}

\begin{notation} 
Let $\bsa:=a_{0},\dots,a_{n}$ be elements in
$H^{*}(G,k)$, of the same positive degree, that are algebraically
independent over $k$.  Let $K$ be the field of rational functions in
indeterminates $\bst:=t_{1},\dots,t_{n}$. Since there is a canonical
isomorphism
\[ 
H^{*}(G_{K},K)\cong H^{*}(G,k) \otimes_{k} K
\] 
as $K$-algebras, we view $H^{*}(G,k)$ as a subring of $H^{*}(G_{K},K)$, and consider elements
\[ 
b_{i}:= a_{i} - a_{0}t_{i}\quad\text{for $i=1,\dots,n$}
\] 
in $H^{*}(G_{K},K)$. Set $d=n|a_{0}|$. Composing the canonical map $k\to K\da_{G}$ with restriction to $G$ of $K\to \Omega^{d} (\kos
K{\bsb})$ in $\StMod G_{K}$ from \eqref{eq:koszul-map}, one gets a morphism
\begin{equation}
\label{eq:fmap} f\colon k \lra \Omega^{d} (\kos K{\bsb})\da_{G}
\end{equation} in $\StMod G$.  Let $U_{\bsa}$ be the multiplicatively
closed set defined in equation~\eqref{eq:localising-set}.
\end{notation}

\begin{theorem}
\label{th:localisation-model} 
With $f$ the map defined in \eqref{eq:fmap}, the following statements hold.
\begin{enumerate}[{\quad\rm(1)}]
\item
The morphism $U_{\bsa}^{-1}f$  is an isomorphism. 
\item
The set $U_{\bsa}$ is invertible on $U^{-1}_{a_{0}} \Omega^{d}(\kos K{\bsb})\da_{G}$.
\item
For any $kG$-module $M$, the natural map $\Omega^{|a_{0}|} M\xra{a_{0}}  M$ becomes an isomorphism when localised at $U_{a_{0}}$.
\end{enumerate}
Consequently, in $\StMod G$  there are isomorphisms
\[ 
\xymatrixcolsep{1.2pc}
\xymatrix{ 
U_{\bsa}^{-1}k \ar@{->}[rr]_-{U_{\bsa}^{-1}f}^-{\cong} 
	&&  U^{-1}_{\bsa} \Omega^{d} (\kos K{\bsb})\da_{G}
		& \ar@{->}[l]_-{\cong} U^{-1}_{a_{0}}\Omega^{d} (\kos K{\bsb})\da_{G} \ar@{->}[r]^-{\cong}
		    & U^{-1}_{a_{0}} (\kos K{\bsb})\da_{G}}.
\] 

\end{theorem}

The proof takes a little preparation. Given a $G$-module $M$, we write $\Loc_G(M)$ for the smallest localising subcategory of $\StMod G$ that contains $M$, and $\Loc_G^\otimes(M)$ for the smallest tensor ideal localising subcategory of $\StMod G$ containing $M$.

\begin{lemma}
\label{le:stable-iso} Let $g\colon M\to N$ be a morphism in $\StMod
G$. If $M,N$ are in $\Loc_{G}(k)$ and $\sHom^{*}_{G}(k,g)$ is an
isomorphism, then so is $g$.
\end{lemma}

\begin{proof} Let $C$ be the cone of $g$ in $\StMod G$; the hypotheses
is that $\sHom^{*}_{G}(k,C)=0$. Since $C$ is also in $\Loc_{G}(k)$, it
follows that it is zero in $\StMod G$, and hence that $g$ is an
isomorphism.
\end{proof}

The result below is well-known, and is recalled here for convenience.

\begin{lemma}
\label{le:unstable} For any element $a$ in $H^{*}(G,k)$ of positive
degree and $G$-module $M$, the natural map $\Ext^{*}_{G}(k,M)\to
\sHom^{*}_{G}(k,M)$ induces an isomorphism
\[ U^{-1}_{a} \Ext^{*}_{G}(k,M)\xra{\ \cong\ } U^{-1}_{a}\sHom^{*}_{G}(k,M)
\]
\end{lemma}

\begin{proof} 
The main point is that there is an exact sequence
\[ 
0\lra {\rm PHom}_{G}(k,M) \lra \Ext^{*}_{G}(k,M) \lra \sHom_{G}^{*}(k,M) \lra C \lra 0
\] 
of graded $H^{*}(G,k)$-modules, where $C$ is concentrated in negative degrees; see, for example, \cite[Section~2]{Benson/Krause:2002a}.  For degree reasons, it is clear that ${\rm PHom}_{G}(k,M)$ and $C$ are torsion with respect to $H^{\ges 1}(G,k)$, and so are annihilated when $a$ is inverted.
\end{proof}

The next result concerns weak sequences; see
Definition~\ref{de:weakly-regular}.

\begin{lemma}
\label{le:Koszul-regular} When $\bsb:=b_{1},\dots,b_{n}$ is a weak
$U_{a}^{-1} H^{*}(G,M)$-sequence for some element $a$ in $H^{*}(G,k)$,
the natural map $M\to \Omega^{d}\kos M{\bsb}$, where $d=\sum_{i=1}^{n}
|b_{i}|$, induces an isomorphism of graded $H^{*}(G,k)$-modules
\[ U_{a}^{-1} \frac {H^{*}(G,M)}{\bsb\, H^{*}(G,M)} \xra{\ \cong\ }
U_{a}^{-1} H^{*}(G, \Omega^{d}\kos M{\bsb}).
\]
\end{lemma}

\begin{proof} 
  It suffices to verify the claim for $n=1$; the general case follows
  by iteration. The exact triangle
  $\Omega^{d}M \xra{b} M \to \Omega^{d}\kos M{b}\to$ induces an exact
  sequence
\[ 
0\lra \frac{H^{*}(G,M)}{b H^{*}(G,M)} \lra H^{*}(G,\Omega^{d}\kos Mb) \lra \Sigma^{d+1}(0: b)\lra 0
\] 
of graded $H^{*}(G,k)$-modules. Here $(0: b)$ denotes the elements of $H^{*}(G,M)$ annihilated by $b$. Localising the sequence above at
$a$ gives the desired isomorphism, since $b$ is not a zerodivisor on $U_{a}^{-1}H^{*}(G,M)$.
\end{proof}

\begin{proof}[Proof of Theorem~\ref{th:localisation-model}] 
By construction, in $\StMod G$ there is an exact triangle 
\[
\Omega^{|a_{0}|}M\lra M\lra \Omega^{|a_0|}(\kos M{a_{0}})\lra\,.
\]
Since $\supp_{G}(\kos M{a_{0}})\subseteq \mcV(a_{0})$, by equation \eqref{eq:kos-support}, one has $U^{-1}_{a_{0}}(\kos M{a_{0}})=0$. Thus, (3) is immediate from the exact triangle above.

As to (1) and (2), set $W:=\Omega^{d}(\kos K{\bsb})$. Since $K\da_{G}$ is a direct sum of copies of $k$, it follows that $W\da_{G}$ is in $\Loc_{G}(k)$. Thus, in view of Lemma~\ref{le:stable-iso}, it suffices to prove that the morphism $f\colon k\to W\da_{G}$ induces an isomorphism
\[ 
U_{\bsa}^{-1} \sHom^{*}_{G}(k,k)\lra U_{a_{0}}^{-1} \sHom^{*}_{G}(k,W\da_{G})
\]
of graded $H^*(G,k)$-modules. Note that this map is isomorphic to
\[ 
U_{\bsa}^{-1} H^{*}(G,k)\lra U_{a_{0}}^{-1} H^{*}(G,W\da_{G})
\]
by Lemma~\ref{le:unstable}, since the degree of elements in $\bsa$ is positive.

As $H^{*}(G_{K},K)\cong H^{*}(G,k)\otimes_{k}K$ it follows from Lemma~\ref{le:localisation}(2) that $\bsb$ is a weak $U_{a_{0}}^{-1} H^{*}(G_{K},K)$-sequence. Thus Lemma~\ref{le:Koszul-regular} gives the first isomorphism below
\[ 
U_{a_{0}}^{-1} \frac{H^{*}(G_{K},K)}{ \bsb H^{*}(G_{K},K)} 
	\xra{\ \cong\ } U_{a_{0}}^{-1} H^{*}(G_{K}, W) \xra{\ \cong\ } U_{a_{0}}^{-1} H^{*}(G, W\da_{G}).
\] 
The second isomorphism is a standard adjunction. It remains to compose this with the isomorphism in Lemma~\ref{le:localisation}(1).
\end{proof}

\begin{notation}
\label{no:generic-point} 
Fix a point $\fp\in\Proj H^{*}(G,k)$, and let $a_{0},\dots,a_{n}$ be elements in $H^{*}(G,k)$ that give a Noether normalisation of $H^{*}(G,k)/\fp$; see Definition~\ref{de:noether-normalisation}.

Let $K$ be the field of rational functions in indeterminates $t_{1},\dots,t_{n}$. Consider the ideal in $H^{*}(G_{K},K)$ given by
\[ 
\fq := (\fp, b_{1},\dots,b_{n})\quad \text{where $b_{i}=a_{i}-a_{0}t_{i}$}\,.
\] 
Then $\fm:=\sqrt{\fq}$ is a closed point in $\Proj H^{*}(G_{K},K)$ lying over $\fp$, by Theorem~\ref{th:generic-point}.

Choose a finite set $\bsp\subseteq \fp$ that generates the ideal $\fp$, let $\bsq:=\bsp\cup\bsb$, and set
\[
\kos k{\fp}:= \kos k{\bsp} \quad \text{ and } \quad \kos K{\fq}:= \kos K{\bsq}\,.
\]
The $G$-module $\kos k{\fp}$ depends on the choice of $\bsp$; however, the thick subcategory of $\stmod kG$ generated by it is independent of the choice; this can be proved along the lines of \cite[Corollary~5.11]{Hovey/Palmieri/Strickland:1997a}. The next result holds for any choice of $\bsp$, and the corresponding $\bsq$.
\end{notation}

\begin{theorem}
\label{th:residue-model} 
The $G$-module $({\kos K{\fq}})\da_{G}$ is $\fp$-local and $\fp$-torsion and $f\otimes_{k} \kos k{\fp}$, with $f$ as in \eqref{eq:fmap}, induces an isomorphism $(\kos k{\fp})_{\fp}\cong ({\kos K{\fq}})\da_{G}$. Thus  in $\StMod G$ there is a commutative diagram
\[ 
\xymatrixcolsep{1.5pc} 
\xymatrix{ &\kos k{\fp}
\ar@{->}[dr]^-{f\otimes_{k}\kos k{\fp}} \ar@{->}[dl] \\
(\kos k{\fp})_{\fp} \ar@{->}[rr]_-{\cong} && (\kos K{\fq})\da_{G}
}
\] 
where the map pointing left is localisation. In particular, there is an equality
\[ 
\gam_{\fp}(\StMod G) = \Loc^{\otimes}_{G}(({\kos K{\fq}})\da_{G})\,.
\]
\end{theorem}

\begin{proof} 
Since $\supp_{G_{K}}(\kos K{\fq})$ equals $\{\fm\}$, by equation \eqref{eq:kos-support}, it follows from the construction of $\fm$ and Proposition~\ref{pr:basechange-supp} that $\supp_{G}(\kos K{\fq})\da_{G}$ equals $\{\fp\}$. Said otherwise, $(\kos K{\fq})\da_{G}$ is $\fp$-local and $\fp$-torsion, as claimed.

Set $W:=(\kos K{\bsb})\da_{G}$. Observe that the restriction functor $(-)\da G$ is compatible with construction of Koszul objects with respect to elements of $H^{*}(G,k)$. This gives a natural isomorphism
\[ 
\kos W{\fp} \cong  (\kos K{\fq})\da_{G}\,.
\] 
Since we already know that the module on the right is $\fp$-local, so is the one on the left. This justifies the last isomorphism below.
\[ 
(\kos k{\fp})_{\fp} \cong \kos{(k_{\fp})}{\fp} \cong \kos {(W_{\fp})}{\fp} \cong (\kos W{\fp})_{\fp}\cong \kos W{\fp}
\] 
The second is the one induced by the isomorphism in Theorem~\ref{th:localisation-model}, since $U_{\bsa}$ is not contained in $\fp$.  The other  isomorphisms are standard.  The concatenation of the isomorphisms is the one in the statement of the theorem.

By \cite[Corollary~8.3]{\bik:2008a} one has $\gam_{\fp}k\otimes_{k}N\cong N$ for any $\fp$-local and $\fp$-torsion $G$-module $N$. This justifies the first equality below.
\[ 
\gam_{\fp}(\StMod G) = \Loc^{\otimes}_{G}(\gam_{\fp}k) = \Loc^{\otimes}_{G}((\kos k{\fp})_{\fp})\,.
\] 
For the second one see, for example, \cite[Lemma~3.8]{\bik:2011a}. Thus, the already established part of the theorem gives the desired equality.
\end{proof}

\begin{lemma}
\label{le:building} Let $K/k$ be an extension of fields and $M$ a
$G$-module. If a $G_{K}$-module $N$ is in
$\Loc^{\otimes}_{G_{K}}(M_{K})$, then $N\da_{G}$ is in
$\Loc^{\otimes}_{G}(M)$.
\end{lemma}

\begin{proof} 
Let $S(G)$ denote a direct sum of a representative set of simple $G$-modules. Then $\Loc_G^\otimes(M)=\Loc_G(S(G)\otimes_k M)$. Note
that $S(G_K)$ is a direct summand of $S(G)_K$.  Now suppose that
\[
N\in \Loc^{\otimes}_{G_{K}}(M_{K})=\Loc_{G_{K}}(S(G)_K\otimes_K  M_{K})\,.
\] 
Since there is an isomorphism of $G$-modules
\[
(S(G)_K\otimes_K  M_{K})\da_G\cong S(G)\otimes_k (M_{K})\da_G\,,
\]
one gets the following 
\[
N\da_G\in \Loc_G(S(G)\otimes_k (M_{K})\da_G)=\Loc_G^\otimes((M_{K})\da_G)=\Loc_G^\otimes(M)\,,
\]
where the last equality uses that $ (M_{K})\da_G$ equals a direct sum of copies of $M$.
\end{proof}

\begin{theorem}
\label{thm:minimality} 
Let $G$ be a finite group scheme over $k$. The tensor triangulated category $\gam_{\fp}(\StMod G)$ is minimal for
each point $\fp$ in $\Proj H^{*}(G,k)$.
\end{theorem}

\begin{proof}
Given the description of $\gam_{\fp}(\StMod G)$ in Theorem~\ref{th:residue-model}, it suffices to verify that if $\fp$ is in the support of a $G$-module $M$, then $(\kos K{\fq})\da_{G}$ is in $\Loc^{\otimes}_{G}(M)$. Let $K/k$ be the extension of fields and $\fm$ the closed point of $\Proj H^{*}(G_{K},K)$ lying over $\fp$ constructed in \ref{no:generic-point}. Then $\supp_{G_{K}}(M_{K})$ contains $\fm$, by Proposition~\ref{pr:basechange-supp}. By \eqref{eq:kos-support}, $\supp_{G_{K}}(\kos K{\fq})=\mcV(\fq)=\{\fm\}$ so Proposition~\ref{pr:minimality-closed} implies $\kos K{\fq}$ is in $\Loc^{\otimes}_{G_{K}}(M_{K})$. It follows from Lemma~\ref{le:building} that $(\kos K{\fq})\da_{G}$ is in $\Loc^{\otimes}_{G}(M)$.
\end{proof}

\part{Applications}
\label{part:applications}
The final part of this paper is devoted to applications of the results
proved in the preceding part. We proceed in several steps and derive
global results about the module category of a finite group scheme from
local properties.

As before, $G$ denotes a finite group scheme over a field $k$ of positive characteristic.

\section{Cosupport equals $\pi$-cosupport}
\label{se:cosupport}
In this section we show that $\pi$-cosupport of any $G$-module
coincides with its cosupport introduced in Section~\ref{se:bik}. The
link between them is provided by a naturally defined $G$-module,
$\alpha_{*}(K)\da_G = (KG\otimes_{K[t]/(t^{p})}K)\da_{G}$, that is the subject of the result below.  For
its proof we recall~\cite[I.8.14]{Jantzen:2003a} that given any
subgroup scheme $H$ of $G$ there is a functor
\[ 
\coind^G_H\colon \Mod H \lra \Mod G.
\]
that is \emph{left} adjoint to the restriction functor $(-)\da_{H}$ from $\Mod G$ to $\Mod H$.

\begin{lemma}
\label{le:supp-alpha_*} 
Fix a point $\fp$ in $\Proj H^{*}(G,k)$. If $\alpha\colon K[t]/(t^{p})\to KG$ is a $\pi$-point corresponding to
$\fp$, then $\supp_G(\alpha_*(K)\da_{G})=\{\fp\}$ holds.
\end{lemma}

\begin{proof} 
We proceed in several steps. Suppose first that $K=k$ and that $G$ is unipotent. Since $\alpha_*(k)$ is a finite dimensional $k$-vector space $\supp_G(\alpha_*(k))$ coincides with the set of prime ideals in $\Proj H^*(G,k)$ containing the annihilator of the $H^*(G,k)$-module $\Ext_{G}^*(\alpha_*(k),\alpha_*(k))$; see \cite[Theorem~5.5]{\bik:2008a}. This annihilator coincides with that of $\Ext_{G}^*(\alpha_*(k),k)$, since $G$ is unipotent, where $H^*(G,k)$ acts via the canonical map $H^*(G,k)\to\Ext_{G}^*(\alpha_*(k),\alpha_*(k))$. Adjunction yields an
isomorphism
\[ 
\Ext_{G}^*(\alpha_*(k),k)\cong \Ext_{k[t]/(t^p)}^*(k,k)
\] 
and we see that the action of $H^*(G,k)$ factors through the canonical map
\[ 
H^*(\alpha)\colon H^*(G,k)\lra H^*(k[t]/(t^p),k)
\] 
that is induced by restriction via $\alpha$. Thus the annihilator of $\Ext_{G}^*(\alpha_*(k),\alpha_*(k))$ has the same radical as $\Ker
H^*(\alpha)$, which is $\fp$. It follows that $\supp_G(\alpha_{*}(k))=\{\fp\}$.

Now let $\alpha$ be arbitrary. We may assume that it factors as
\[ 
K[t]/(t^p) \xra{\ \beta\ } KU \lra KG
\] 
where $U$ is a quasi-elementary subgroup scheme of $G_{K}$; see Remark~\ref{re:pi-basics}(2). Note that $\beta$ defines a $\pi$-point
of $U$; call it $\fm$. The first part of this proof yields an equality
\[ 
\supp_U(\beta_{*}(K)) = \{\fm\}\,.
\] 
Let $f\colon H^*(G_{K},K) \to H^*(U,K)$ be the restriction map and $\phi$ the map it induces on $\Proj$. Note that $\phi(\fm)$ is the
$\pi$-point of $G_{K}$ corresponding to $\alpha$. Therefore, applying Proposition~\ref{pr:change-cat-ring} to the pair
\[ 
((-)\da_{U},f)\colon \StMod G_{K} \to \StMod U\,,
\] 
one gets the inclusion below
\[ 
\supp_{G_{K}}(\alpha_{*}(K)) = \supp_{G_{K}}(\coind_U^{G_{K}}(\beta_{*}(K)))\subseteq \phi(\supp_U(\beta_{*}K)) = \{\phi(\fm)\}.
\]
Since $\Ext^{*}_{G_{K}}(\alpha_{*}(K),K)$ is non-zero, by adjointness, $\alpha_{*}(K)$ is not projective. Thus its support equals $\{\phi(\fm)\}$.  It remains to apply Proposition~\ref{pr:basechange-supp}(2).
\end{proof}

\begin{lemma}
\label{le:pi-cosupp-alpha_*} Let $\alpha\colon K[t]/(t^{p})\to KG$ be
a $\pi$-point corresponding to a point $\fp$ in $\Proj H^{*}(G,k)$, and
$M$ a $G$-module. The following conditions are equivalent.
\begin{enumerate}[\quad\rm(1)]
\item $\fp$ is in $\picosupp_{G}(M)$;
\item $\Hom_k(\alpha_*(K)\da_{G},M)$ is not projective;
\item $\sHom_{G}(\alpha_*(K)\da_{G},M)\ne 0$.
\end{enumerate}
\end{lemma}

\begin{proof} The equivalence of (1) and (3) follows from the
  definition of $\pi$-cosupport and the following standard adjunction
  isomorphisms
\[ \sHom_{K[t]/(t^p)}(K,\alpha^*(M^K)) \cong
\sHom_{G_K}(\alpha_*(K),M^K) \cong \sHom_{G}(\alpha_*(K)\da_G,M)
\]

(1)$\iff$(2) Let $S$ be the direct sum of a representative set of
simple $kG$-modules. Since $\pisupp_{G}(S)$ equals $\Proj H^{*}(G,k)$,
Theorem~\ref{th:tensor-and-hom-pi} yields an equality
\[ \picosupp_{G}(M)=\picosupp_{G}(\Hom_k(S,M))\,.
\] This justifies the first of the following equivalences.
\begin{align*} \fp\in\picosupp_{G}(M)\quad&\iff\quad
\fp\in\picosupp_{G}(\Hom_k(S,M))\\ &\iff\quad
\sHom_{G}(\alpha_*(K)\da_G,\Hom_k(S,M))\neq 0\\ &\iff\quad
\sHom_{G}(\alpha_*(K)\da_G\otimes_k S,M)\neq 0\\ &\iff\quad
\sHom_{G}(S,\Hom_k(\alpha_*(K)\da_{G},M))\neq 0\\ &\iff\quad
\Hom_k(\alpha_*(K)\da_{G},M) \text{ is not projective}.
\end{align*} The second one is (1)$\iff$(3) applied to
$\Hom_{k}(S,M)$; the third and the fourth are standard adjunctions,
and the last one is clear.
\end{proof}

\begin{theorem}
\label{th:picosupp=bik} Let $G$ be a finite group scheme over a field
$k$. Viewed as subsets of $\Proj H^*(G,k)$ one has
$\picosupp_G(M)=\cosupp_G(M)$ for any $G$-module $M$.
\end{theorem}

\begin{proof} 
The first of the following equivalences is Lemma~\ref{le:pi-cosupp-alpha_*}. 
\begin{align*} 
\fp\in\picosupp_{G}(M)\quad
&\iff\quad \Hom_k(\alpha_*(K)\da_{G},M) \text{ is not projective}\\ 
&\iff\quad \Hom_k(\gam_\fp k,M) \text{ is not projective}\\
&\iff\quad\fp\in\cosupp_{G}(M).
\end{align*} 
The second one holds because $\alpha_*(K)\da_{G}$ and $\gam_\fp k$ generate the same tensor ideal localising subcategory of $\StMod G$. This is a consequence of Theorem~\ref{thm:minimality} because $\supp_G(\alpha_*(K)\da_{G})=\{\fp\}$ by
Lemma~\ref{le:supp-alpha_*}. The final equivalence is simply the definition of cosupport.
\end{proof}

Here is a first consequence of this result. We have been unable to
verify the statement about maximal elements directly, except for
closed points in the $\pi$-support and $\pi$-cosupport; see
Lemma~\ref{le:maxsupp=maxcosupp}.

\begin{corollary}
\label{co:maxsupp=maxcosupp} For any $G$-module $M$ the maximal
elements, with respect to inclusion, in $\picosupp_{G}(M)$ and
$\pisupp_{G}(M)$ coincide. In particular, $M$ is projective if and
only if $\picosupp_{G}(M)=\varnothing$.
\end{corollary}

\begin{proof} Given Theorems~\ref{th:pisupp=bik} and
\ref{th:picosupp=bik}, this is a translation of
\cite[Theorem~4.5]{\bik:2012b}.
\end{proof}

The next result describes support and cosupport for a subgroup scheme
$H$ of $G$; this complements Proposition~\ref{pr:basechange-supp}. 

Recall that induction and coinduction are related as follows
\begin{equation}
\label{eq:character} 
\ind_H^G(M) \cong \coind_H^G(M \otimes_k \mu)\,,
\end{equation}
with $\mu$ the character of $H$ dual to $(\delta_G)\da_{H}\delta^{-1}_H$, where $\delta_G$ is a linear character of $G$ called the 
\emph{modular function}; see \cite[Proposition~I.8.17]{Jantzen:2003a}.

\begin{proposition}
\label{pr:ind} 
Let $H$ be subgroup scheme of a finite group scheme $G$ over $k$ and
$\rho\colon \Proj H^*(H,k) \to \Proj H^*(G,k)$ 
the map induced by restriction.
 \begin{enumerate}[\quad\rm(1)]
\item For any $G$-module $N$ the following equalities hold
\[ \supp_H(N\da_{H}) = \rho^{-1}(\supp_G(N)) \quad\text{and}\quad
\cosupp_H(N\da_{H}) = \rho^{-1}(\cosupp_G(N))
\]
\item For any $H$-module $M$ the following inclusions hold.
\[ \supp_G(\ind_H^G M) \subseteq \rho(\supp_H(M)) \quad\text{and}\quad
\cosupp_G(\ind_H^G M) \subseteq \rho(\cosupp_H(M)).
\] 
\noindent These become equalities when $G$ is a finite group or $H$
is unipotent.
\end{enumerate}
\end{proposition}
\begin{proof}
 (1) Since any $\pi$-point of $H$ induces a $\pi$-point of $G$, the stated equalities are clear when one replaces support and
cosupport by $\pi$-support and $\pi$-cosupport, respectively. It remains to recall Theorems ~\ref{th:pisupp=bik} and
\ref{th:picosupp=bik}.

(2)  Since $\ind$ is right adjoint to restriction, the inclusion of cosupports is a  consequence of Proposition~\ref{pr:change-cat-ring} applied to the functor
\[ 
((-)\da_{H},f)\colon (\StMod G, H^{*}(G,k)) \lra (\StMod H, H^{*}(H,k))\,,
\]  
where $f\colon H^{*}(G,k)\to H^{*}(H,k)$ is the homomorphism of $k$-algebras induced by restriction. By the same token, as coinduction is left adjoint to restriction one gets
\[
\supp_G(\coind_H^G M) \subseteq \rho(\supp_H(M))\,.
\] 
By equation \eqref{eq:character}, there is a one-dimensional representation $\mu$ of $H$ such that
\[
\supp_G(\ind_H^G M) = \supp_G(\coind_H^G (M \otimes_k \mu))\,.
\]  
This yields the inclusion below.
\[
\supp_G(\ind_H^G M) \subseteq \rho(\supp_H(M \otimes_k \mu)) = \rho(\supp_H(M) \cap \supp_H(\mu)) = \rho(\supp_H(M)).
\]
The first equality is by Theorem~\ref{th:tensor-and-hom-pi} while the second one holds because the support of any one-dimensional representation $\mu$ is $\Proj H^{*}(G,k)$, as follows, for example, because $\Hom_{k}(\mu,\mu)$ is isomorphic to $k$.

Concerning the equalities, the key point is that under the additional hypotheses $\ind_H^G(-)$, which is right adjoint to $(-)\da_{H}$, is faithful on objects.
\end{proof}

\begin{example} 
One of many differences between finite groups and connected group schemes is that Proposition~\ref{pr:ind}(2) may fail
for the latter, because induction is not faithful on objects in general.

For example, let $\mcG = \SL_{n}$ and $\mcB$ be its standard Borel subgroup. Take $G = \mcG_{(r)}$, $H = \mcB_{(r)}$, and $\lambda = \rho (p^r-1)$ where $\rho$ is the half sum of all positive roots for the root system of $\mcG$.  Let $k_\lambda$ be the one-dimensional
representation of $H$ given by the character $\lambda$. Then $\ind_{H}^{G}k_\lambda$ is the \emph{Steinberg module} for $G$; in particular, it is projective.  Hence, $\ind\colon \stmod H \to \stmod G$ is not faithful on objects, and both inclusions of Proposition~\ref{pr:ind}(2) are strict for $M=k_{\lambda}$.
\end{example}

\section{Stratification} 
\label{se:stratification}
In this section we establish for a finite group scheme the
classification of tensor ideal localising subcategories of its stable
module category and draw some consequences. The development follows
closely the one in \cite[Sections~10 and 11]{\bik:2011b}. For this
reason, in the remainder of the paper, we work exclusively with
supports as 
defined in Section~\ref{se:bik}, secure in the knowledge afforded by
Theorem~\ref{th:pisupp=bik} that the discussion could just as well be
phrased in the language of $\pi$-points.

\begin{theorem}
\label{th:stratification} 
Let $G$ be a finite group scheme over a field $k$. Then the stable module category $\StMod G$ is stratified as a tensor triangulated category by the natural action of the cohomology ring $H^*(G,k)$. Therefore the assignment
\begin{equation}
\label{eq:supp} 
\sfC\longmapsto \bigcup_{M\in\sfC}\supp_G(M)
\end{equation} 
induces a bijection between the tensor ideal localising subcategories of $\StMod G$ and the subsets of $\Proj H^*(G, k)$.
\end{theorem}

\begin{proof} The first part of the assertion is precisely the
statement of Theorem~\ref{thm:minimality}.  The second part of the
assertion is a formal consequence of the first; see
\cite[Theorem~3.8]{Benson/Iyengar/Krause:2011b}. The inverse map sends
a subset $\mcV$ of $\Proj H^*(G, k)$ to the tensor ideal localising
subcategory consisting of all $G$-modules $M$ such that
$\supp_G(M)\subseteq\mcV$.
\end{proof}

The result below contains the first theorem from the introduction.

\begin{corollary}
\label{co:modules} Let $M$ and $N$ be non-zero $G$-modules. One can
build $M$ out of $N$ if (and only if) there is an inclusion
$\pisupp_{G}(M)\subseteq \pisupp_{G}(N)$.
\end{corollary}

\begin{proof} 
  The canonical functor $\Mod G \to \StMod G$ that assigns a module to
  itself respects tensor products and takes short exact sequences to
  exact triangles. It follows that $M$ is built out of $N$ in $\Mod G$
  if and only if $M$ is in the tensor ideal localising subcategory of $\StMod G$
  generated by $N$; see also
  \cite[Proposition~2.1]{Benson/Iyengar/Krause:2011b}. The desired
  result is thus a direct consequence of
  Theorem~\ref{th:stratification}.
\end{proof}

\subsection*{Thick subcategories} 
As a 
corollary of Theorem~\ref{th:stratification} we deduce a
classification of the tensor ideal thick subcategories of $\stmod G$,
stated already in \cite[Theorem~6.3]{Friedlander/Pevtsova:2007a}. The
crucial input in the proof in \textit{op.\ cit.}\ is
\cite[Theorem~5.3]{Friedlander/Pevtsova:2007a}, which is flawed (see
Remark~\ref{re:fpnot}) but the argument can be salvaged by referring
to Theorem~\ref{th:main} instead. We give an alternative proof,
mimicking \cite[Theorem~11.4]{\bik:2011b}.

\begin{theorem}
\label{th:thick} 
Let $G$ be a finite group scheme over a field $k$. The assignment \eqref{eq:supp} induces a bijection between tensor ideal thick subcategories of $\stmod G$ and specialisation closed subsets of $\Proj H^*(G,k)$.
\end{theorem}

\begin{proof} 
  To begin with, if $M$ is a finite dimensional $G$-module, 
  then $\supp_G(M)$ is a Zariski closed subset of $\Proj H^*(G,k)$;
  conversely, each Zariski closed subset of $\Proj H^*(G,k)$ is of
  this form. Indeed, given the identification of $\pi$-support and
  cohomological support, the forward implication statement follows
  from \cite[Proposition~3.4]{Friedlander/Pevtsova:2007a} while the
  converse is
  \cite[Proposition~3.7]{Friedlander/Pevtsova:2007a}. Consequently, if
  $\sfC$ is a tensor ideal thick subcategory of $\stmod G$, then
  $\supp_{G}(\sfC)$ is a specialisation closed subset of
  $\Proj H^*(G,k)$, and every specialisation closed subset of
  $\Proj H^*(G,k)$ is of this form. It remains to verify that the
  assignment $\sfC\mapsto \supp_{G}(\sfC)$ is one-to-one.

This can be proved as follows: $\StMod G$ is a compactly generated
triangulated category and the full subcategory of its compact objects
identifies with $\stmod G$. Thus, if $\sfC$ is a tensor ideal thick
subcategory of $\stmod G$ and $\sfC'$ the tensor ideal localising
subcategory of $\StMod G$ that it generates, then
$\sfC'\cap \stmod G=\sfC$; see \cite[\S5]{Neeman:1996a}. Since
$\supp_{G}(\sfC')=\supp_{G}(\sfC)$, Theorem~\ref{th:stratification}
gives the desired result.
\end{proof}

\subsection*{Localising subcategories closed under products} The
following result describes the localising subcategories of $\StMod
 G$ that are closed under products.

\begin{theorem}
\label{th:products} A tensor ideal localising subcategory of $\StMod
 G$ is closed under products if and only if the complement of its
support in $\Proj H^*(G,k)$ is specialisation closed.
\end{theorem}

\begin{proof}
For the case that $kG$ is the group algebra of a finite group, see \cite[Theorem~11.8]{\bik:2011b}.  The argument applies
verbatim to finite group schemes; the main ingredient is the stratification of $\StMod G$, Theorem~\ref{th:stratification}.
\end{proof}

\subsection*{The telescope conjecture} 
A localising subcategory of a 
compactly generated triangulated category $\sfT$ is \emph{smashing} if
it arises as the kernel of a localisation functor $\sfT\to\sfT$ that
preserves coproducts. The telescope conjecture, due to Bousfield and
Ravenel \cite{Bousfield:1979a,Ravenel:1984a}, in its general form is
the assertion that every smashing localising subcategory of $\sfT$ is
generated by objects that are compact in $\sfT$; see
\cite{Neeman:1992b}. The following result confirms this conjecture for
$\StMod G$, at least for all smashing subcategories that are tensor
ideal. Note that when the trivial $kG$-module $k$ generates $\stmod G$
as a thick subcategory (for example, when $G$ is unipotent) each
localising subcategory is tensor ideal.

\begin{theorem}
\label{th:smashing} 
Let $\sfC$ be a tensor ideal localising subcategory of $\StMod  G$. The following conditions are equivalent:
\begin{enumerate}[\quad\rm(1)]
\item The localising subcategory $\sfC$ is smashing.
\item The localising subcategory $\sfC$ is generated by objects compact in $\StMod  G$.
\item The support of $\sfC$ is a specialisation closed subset of
$\Proj H^*(G, k)$.
\end{enumerate}
\end{theorem}

\begin{proof} 
If $G$ is a finite group this result is \cite[Theorem~11.12]{\bik:2011b} and is deduced from the special case of Theorem~\ref{th:stratification} for finite groups. 
The proof carries over verbatim to group schemes.
\end{proof}

\subsection*{The homotopy category of injectives} 
Let $\KInj{G}$ denote the triangulated category whose objects are the complexes of injective $G$-modules and whose morphisms are the homotopy classes of degree preserving maps of complexes.  As a triangulated category $\KInj{G}$ is compactly generated, and the compact objects are equivalent to $\sfD^{\sfb}(\mod G)$, via the functor $\KInj{G}\to \sfD(\Mod G)$. The tensor product of modules extends to complexes and defines a tensor product on $\KInj{G}$. This category was investigated in detail by Benson and Krause \cite{Benson/Krause:2008a} in case $G$ is a finite group; the more general case of a finite group scheme is analogous. Taking Tate resolutions gives an equivalence of triangulated categories from the stable module category $\StMod G$ to the full subcategory $\KacInj{G}$ of $\KInj{G}$ consisting of acyclic complexes. This equivalence preserves the tensor product. The Verdier quotient of $\KInj{G}$ by $\KacInj{G}$ is equivalent, as a triangulated category, to the unbounded derived category $\sfD(\Mod G)$. There are left and right adjoints, forming a recollement
\[ 
\StMod G \xrightarrow{\sim} \KacInj{G} \ 
\begin{smallmatrix}
- \otimes_k tk  \\ \hbox to 50pt{\leftarrowfill} \\ \hbox to
50pt{\rightarrowfill} \\ \hbox to 50pt{\leftarrowfill} \\ \Hom_k(tk,-)
\end{smallmatrix}
\ \KInj{G} \ 
\begin{smallmatrix} 
- \otimes_k pk \\
\hbox to 50pt{\leftarrowfill} \\ \hbox to 50pt{\rightarrowfill} \\
\hbox to 50pt{\leftarrowfill} \\ \Hom_k(pk,-) 
\end{smallmatrix} 
\  \sfD(\Mod G)
\] 
where $pk$ and $tk$ are a projective resolution and a Tate resolution of $k$ respectively.

The cohomology ring $H^{*}(G,k)$ acts on $\KInj{G}$ and, as in \cite{\bik:2008a,\bik:2012b}, the theory of supports and cosupports for $\StMod G$ extends in a natural way to $\KInj{G}$. It associates to each $X$ in $\KInj{G}$ subsets $\supp_{G}(X)$ and $\cosupp_{G}(X)$ of $\Spec H^{*}(G,k)$.  The Tate resolution of a $G$-module $M$ is $tk\otimes_{k}M$, so there are equalities
\[ 
\supp_{G}(M) = \supp_{G}(tk\otimes_{k}M)\quad\text{and}\quad
\cosupp_{G}(M) = \cosupp_{G}(tk\otimes_{k}M),
\] 
where one views $\Proj H^{*}(G,k)$ as a subset of $\Spec H^{*}(G,k)$. Thus Theorem~\ref{th:stratification} has the following
consequence.

\begin{corollary}
\label{co:stratification} 
The homotopy category $\KInj{G}$ is stratified as a tensor triangulated category by the natural action of
the cohomology ring $H^*(G,k)$.  Therefore the assignment $\sfC\mapsto
\bigcup_{X\in\sfC}\supp_{G}(X)$ induces a bijection between the tensor
ideal localising subcategories of $\KInj{G}$ and the subsets of
$\Spec H^*(G, k)$. It restricts to a bijection between the tensor
ideal thick subcategories of $\sfD^{\sfb}(\mod G)$ and specialisation
closed subsets of $\Spec H^*(G, k)$.  \qed
\end{corollary}

With this result on hand, one can readily establish analogues of Theorems~\ref{th:products} and \ref{th:smashing} for $\KInj{G}$. We
leave the formulation of the statements and the proofs to the interested reader; see also \cite[Sections~10 and 11]{\bik:2011b}.

To wrap up this discussion, we record a proof of the following result mentioned in the introduction.

\begin{theorem}
\label{th:last}
If $G$ is unipotent and $M,N$ are $G$-modules, then $\Ext^{i}_{G}(M,N)=0$ for some $i\ge 1$ if only if $\pisupp_{G}M\cap \picosupp_{G}N=\varnothing$; when these conditions hold, $\Ext^{i}_{G}(M,N)=0$ for all $i\ge 1$.
\end{theorem}

\begin{proof} Let $L = \Hom_{k}(M,N)$. Then Theorem~\ref{th:tensor-and-hom-pi}(2) and Corollary~\ref{co:maxsupp=maxcosupp} imply that $\pisupp_{G}M\cap \picosupp_{G}N=\varnothing$ if and only if $\picosupp_{G}L = \varnothing$ if and only if $L$ is projective. If $L$ is projective, then $\Ext_{G}^i(M,N) \cong \Ext_{G}^i(k, L) = 0$ for $i>0$. 

It remains to show that if $\Ext_{G}^i(k, L) = 0$ for some $i>0$, then $L$ is projective.  Indeed, since $G$ is unipotent, $k$ is the only simple $G$-module. Hence, the condition $\Ext_{G}^i(k, L) = 0$ implies that $\Ext_{G}^i(N,L) =0$ for any finite dimensional $G$-module $N$ since it has a finite filtration with all subquotients isomorphic to $k$. Therefore, $\Ext_{G}^1(N, \Omega^{1-i}L) \cong \Ext_{G}^i(N,L) =0$ for any finite dimensional $G$-module $N$. We conclude that $\Omega^{1-i}L$ is projective and hence that $L$ is projective. 
\end{proof}

It is immediate from definitions that $\pisupp_{G}(M)=\picosupp_{G}(M)$ for any finite dimensional $G$-module $M$. Thus, the result above implies that when $M$ and $N$ are finite dimensional, one has
\[
\Ext^{i}_{G}(M,N)=0 \quad \text{for $i\gg 0$} \qquad \iff \qquad  \Ext^{i}_{G}(N,M)=0 \quad\text{for $i\gg 0$}\,.
\]
This can be verified directly, using  the results from \cite{Friedlander/Pevtsova:2007a} pertaining only to finite dimensional $G$-modules.

\begin{ack} 
The authors are grateful to Amnon Neeman for a careful reading of this paper and in particular for pointing out an error in an earlier version of Theorems~\ref{th:unip} and \ref{th:generic-point}. Part of this article is based on work supported by the National Science Foundation under Grant No.\,0932078000, while DB, SBI, and HK were in residence at the Mathematical Sciences Research Institute in Berkeley, California, during the 2012--2013 Special Year in Commutative Algebra.  The authors thank the Centre de Recerca Matem\`atica, Barcelona, for hospitality during a visit in April 2015 that turned out to be productive and pleasant.
\end{ack}

\bibliographystyle{amsplain}
\newcommand{\noopsort}[1]{}
\providecommand{\MR}{\relax\ifhmode\unskip\space\fi MR }
\providecommand{\MRhref}[2]{%
  \href{http://www.ams.org/mathscinet-getitem?mr=#1}{#2}}
\providecommand{\href}[2]{#2}

\end{document}